\documentclass[notitlepage,reqno,11pt]{amsart}
\usepackage{amsmath}
\usepackage{amssymb}
\usepackage{amsfonts,amsthm,hyperref}
\usepackage[margin=1in]{geometry}
\usepackage{xcolor}
\usepackage[normalem]{ulem}
\usepackage[foot]{amsaddr}
\usepackage{mathtools}

\usepackage{accents}

\usepackage{soul}

\usepackage{enumitem}

\numberwithin{equation}{section}
\newtheorem{assumption}{Assumption}[section]
\newtheorem{lemma}{Lemma}[section]
\newtheorem{theorem}{Theorem}[section]
\newtheorem{definition}{Definition}[section]

\newtheorem{prop}{Proposition}[section]
\newtheorem{remark}{Remark}[section]

\newcommand{\RR}{{\mathbb R}}

\newcommand{\NN}{{\mathbb N}}

\definecolor{ISUred}{RGB}{200, 16, 46}

\newcommand{\ttl}{\Large An infinite-dimensional Skorokhod problem   \\[5pt]with a diagonalization approach}

\allowdisplaybreaks

\begin{document}

\title[]{\ttl}

\author{Louis T. Clarke$^*$, Guodong Pang$^*$ and Ruoyu Wu$^{**}$}
\address{$^*$Department of Computational Applied Mathematics and Operations Research,
    George R. Brown School of Engineering and Computing, 
    Rice University,
    Houston, TX 77005}
    \email{lc160,\,gdpang@rice.edu}
\address{$^{**}$Department of Mathematics,
    Iowa State University, Ames, IA 50011
    }
    \email{ruoyu@iastate.edu}
        

\begin{abstract} We study an infinite-dimensional Skorokhod problem on the positive cone of $L_1([0,1])$, where a regulator process constrains a free càdlàg path to remain nonnegative through an oblique reflection mechanism determined by a positive linear integral operator $\boldsymbol{F}$ on $L_1([0,1])$ with spectral radius less than $1$. This constrained process $Z$ remains in the positive cone, and the regulator $Y$ is nondecreasing with respect to the cone order and satisfies a pointwise complementarity condition, increasing only on those space-time regions where the corresponding component of $Z$ vanishes. This formulation extends the classical multidimensional Skorokhod problem on the nonnegative orthant to an infinite-dimensional setting.
Under suitable assumptions on the operator $\boldsymbol{F}$ that guarantee a dominant positive eigenfunction, we proceed through a diagonalization argument to give an alternative representation of the system in which the $L_1$ norm of the linear operator is less than 1. From this characterization, we establish existence and uniqueness of solutions for arbitrary input paths $X\in D([0,T],L_1([0,1]))$. We further show that the associated Skorokhod map is well defined and Lipschitz continuous with respect to both the uniform and Skorokhod $J_1$ topologies. As a consequence, solutions may be obtained as limits of multidimensional Skorokhod problems, providing a constructive characterization of the infinite-dimensional reflection mechanism.
These results provide a framework for the study of constrained stochastic systems with infinitely many interacting components and provide a foundation for reflected stochastic processes evolving in function spaces.
\end{abstract}


\maketitle

\section{Introduction}
\label{sec:intro}

The Skorokhod problem provides a mechanism for constraining a path to remain within a prescribed domain through the action of a regulator process. In the classical $d$-dimensional setting on the nonnegative orthant, the reflected path is represented as
\begin{equation} \label{eq:origskorokhod}
z(t)=x(t)+(I-F)y(t),
\end{equation}
where $x$ is the unconstrained input path on $D([0,T],\RR^d)$, and $y$ is a nondecreasing regulator chosen so that $z$ remains nonnegative, and $I$ is the identity matrix. We refer to $z$ as the reflected process and $y$ as the regulator. The matrix $F$ determines the directions in which the regulating force acts at the boundary, and has spectral radius less than $1$. When $F\equiv0$, the reflection occurs along the coordinate directions and coincides with the standard orthogonal reflection. More generally, the term $(I-F)y$ allows the regulator to act in directions that are not normal to the boundary, leading to an oblique reflection mechanism. This has been extensively studied, see \cite{HarrisonReiman1981reflected, chen1994hierarchical, Whitt2002, ChenYao2001fundamentals, haddad2010pathwise}, and has seen particular application to Jackson networks and other stochastic network and queueing models. 

In \cite{clarke2026growing}, this concept was recently extended to an uncountably infinite dimension, considering paths in $D([0,T],L_1([0,1]))$ and replacing the nonnegative orthant with the positive cone of $L_1([0,1])$. Specifically, given a càdlàg input path $X = \{X_u(t)\}_{u\in[0,1],t\in[0,T]}\in D([0,T],L_1([0,1]))$, we consider the problem of finding a pair of càdlàg processes $(Z,Y)$ satisfying
\begin{equation} \label{eq:infskorokhod}
Z(t) = X(t)+({\bf1} - \boldsymbol{F})Y(t),
\end{equation}
where $\boldsymbol{F}$ is a positive linear integral operator on $L_1([0,1])$ with spectral radius less than $1$, known as the reflection operator. We choose $(Y,Z)$ such that the reflected process $Z$ remains in the positive cone, and the regulator $Y$ is nondecreasing with respect to the cone order and satisfies a pointwise complementarity condition, namely, 
$Z_u(t)dY_u(t) = 0$ for all $u\in[0,1]$. Futhermore, one can show that under the uniform norm $\|f\|_{T,1} = \int_0^1\sup_{0\le t\le T}|f_u(t)|du$, the solutions to the Skorokhod problem $Y$ and $Z$ are Lipschitz with respect to inputs $\boldsymbol{F}$ and $X$. Together, we refer to the solution $(Y,Z)$ as the reflection map, and ask, given free process $X$ and reflection operator $\boldsymbol{F}$, whether the reflection map exists, is unique, and is Lipschitz continuous with respect to both the free process and reflection operator.

The Skorokhod problem originated in 1961 with the seminal work of Skorokhod \cite{skorokhod1961stochastic}, who introduced diffusion processes in one dimension with reflecting boundaries and established existence and uniqueness results. In this setting, reflection was studied primarily as a mechanism for constructing stochastic processes. This was further studied by McKean in \cite{mckean1963skorohod,mckean1969stochastic}, looking at multi-dimensional diffusions and stochastic boundary problems. The multi-dimensional problem was further studied by Tanaka in \cite{tanaka1979stochastic}, studying reflected Brownian motion with normal reflection. The theory was significantly extended in 1981 by Harrison and Reiman \cite{HarrisonReiman1981reflected}, who considered reflected Brownian motion in the nonnegative orthant and introduced the multidimensional reflection equation \eqref{eq:origskorokhod}. Subsequent developments shifted attention from specific diffusion processes to the deterministic mapping that associates an unconstrained path with its reflected counterpart. The work by Chen and Mandelbaum \cite{chen1991leontief} studied the problem in $D([0,T],\mathbb{R}^d)$ and developed geometric interpretations of the regulator through production-set formulations and further clarified the structure of oblique reflection. This framework was further developed by Whitt \cite[Chapter 14]{Whitt2002}, formalizing the notion of a regulator set. A more abstract perspective was introduced by Dupuis and Ishii \cite{dupuis1991lipschitz}, who studied the Skorokhod map in $C([0,T],\mathbb{R}^d)$ as a deterministic operator on path spaces and established fundamental continuity and Lipschitz properties. 

When studying the finite-dimensional Skorokhod problem posed by \eqref{eq:origskorokhod}, there are two distinct approaches to establish existence, uniqueness, and Lipschitz continuity. In \cite{HarrisonReiman1981reflected}, the solution to the Skorokhod problem is characterized as the unique fixed point of a mapping which is shown to be a contraction through a diagonalization argument on the routing matrix. Here, the diagonal matrix $D$ takes entries from the dominant eigenvector of $F^T$ such that $DFD^{-1}$ has the maximum column sum equal to its spectral radius, specified as being less than $1$. This property allows for direct contraction mapping arguments to be used in the arguments for existence, uniqueness, and Lipschitz continuity. 
Through the introduction of the regulator set, works such as \cite{chen1991leontief} and \cite{Whitt2002} do not rely on a diagonalization argument. Instead, they show that the regulator $y(t)$ is the infimum of the set $\{w\in D([0,T],\mathbb{R}^d) : x + (I-F)w\ge0\}$. Existence is proved by showing that the regulator set is nonempty, and uniqueness follows from the definition of the infimum. This instead relies on the properties of $F$, in particular that as the spectral radius is less than $1$, it must hold for some $k\le d$ that $F^k$ has maximum column sum less than $1$. From this, one can form $k$-stage contraction arguments similar to those of the diagonalized case, from which continuity and Lipschitz properties can be established.

In moving from finite to infinite dimension, one must formalize similar concepts moving from processes in the space $D([0,T],\mathbb{R}^d)$ to $D([0,T],L_1([0,1]))$. 
In \cite[Section 3]{clarke2026growing}, we adapt the latter of the two methods mentioned in the paragraph above. In this case, the regulator set is extended to this new space, namely $\{W\in D([0,T]L_1([0,1])) : X + ({\bf1}-\boldsymbol{F})W\ge0\}$. Here, the solution $Y$ to \eqref{eq:infskorokhod} is given by the pointwise infimum of the regulator set, with $Z$ following its definition in the same equation. One then shows that the regulator can equivalently be written as the unique fixed point of a mapping $\pi$. Under the condition that the (operator) spectral radius of $\boldsymbol{F}$ is less than $1$, one can then apply the fact that the $L_1$ norm of the $k$-fold composition of $\boldsymbol{F}$ must be less than $1$ for some $k\ge1$ to show that $\pi$ is a $k$-stage contraction. Similarly to \cite{Whitt2002} this can then be used to derive Lipschitz properties of the reflection map with respect to both the free process $X$ and the reflection operator $\boldsymbol{F}$.

In this paper, we instead adopt the diagonalization approach first introduced in \cite{HarrisonReiman1981reflected}. For the finite-dimensional Skorokhod problem, to proceed with a diagonalization argument it is sufficient to simply assume that the spectral radius of the transition matrix is less than 1. However, one must take more care when dealing with operators, as the behavior of a general positive diagonal system may not restrict boundedness of the constants associated with Lipschitz conditions to be finite. To be more preceise, if diagonalization causes the sytem to either vanish or explode at certain points, then these points will cause issues when converting back to the original system. This imposes a new set of assumptions on the reflection operator $\boldsymbol{F}$, ensuring that $\boldsymbol{F}$ is both eventually strongly positive and quasi-compact (see Assumption \ref{assumption:generalF}). Under this framework, we establish existence and uniqueness of solutions to the Skorokhod problem in Theorem \ref{thm:ExtUni}. To do this, we introduce a mapping $\pi_{X,F}$ on the function space for the diagonalized system that is shown to have a unique fixed point that satisfies the Skorokhod problem.

We next consider Lipschitz continuity of the reflection map, first in the $\|\cdot\|_{T,1}$ metric. By creating two iterative processes using the mapping $\pi_{X,F}$ so each converges to the fixed point characterized by the regulator of the Skorokhod problem, we prove two results on Lipschitz continuity, first in the free process $X$ in Theorem \ref{thm:DiagLipschitz}, and then in the reflection operator $\boldsymbol{F}$ in Lemma \ref{lem:refbound}. Combining these two results, along with a further result allowing two separate processes to share the same diagonal operator if their reflection operators are close enough in the $L_1$ norm, provides a Lipschitz continuity result in both the free process and the reflection operator in Theorem \ref{thm:doublelipschitzsame}. This is extended to Theorem \ref{thm:seqlipschitz} for a sequence of processes that converge to a limit process. 
We also study the corresponding Lipschitz continuity properties of the refleciton map under the Skorokhod $J_1$ metric (see the metric defined in \eqref{eq:skotop}). To work with this topology, we first present an explicit result on time rescaling of the reflection map, in Proposition \ref{prop:rescale}. This provides the necessary theory to prove results in Lipschitz continuity with this new metric. with respect to the free process in Theorem \ref{thm:skolipfree} and with respect to both the free process and reflection operator in Theorem \ref{thm:skolipboth}.

The Skorokhod problem plays a fundamental role in the study of constrained stochastic systems. A prominent example arises in Jackson queueing networks, where queue lengths must remain nonnegative and the cumulative idle-time process when a station is not serving any customers acts as a regulator preventing excursions below zero. In the case when the spectral radius of the transition matrix is less than one, this represents an open Jackson network, where jobs can enter and leave the system at different nodes. In heavy traffic, for the queue-length and workload processes, fluid limits are typically given as the deterministic solution of the Skorokhod problem, while  diffusion approximations are described by reflected Brownian motions whose sample paths are obtained through solutions of an associated Skorokhod problem, see \cite{Whitt2002}. 

The infinite-dimensional Skorokhod problem can see application arising naturally in systems with a continuum of interacting components, such as measure-valued population models, stochastic transport and storage systems, and large-scale (spatial)
 queueing networks indexed by a continuous parameter. In these settings the state of the system is more naturally represented by a function or measure rather than a finite-dimensional vector, leading to reflection problems on infinite-dimensional spaces. This paper along with \cite{clarke2026growing} can therefore provide a foundation for the analysis of constrained stochastic processes evolving in function spaces and for the study of continuum limits of large interacting systems with certain constraints.

\subsection{Organization of the paper}
The paper is organized as follows. In the next subsection,  we summarize the notation and some definitions used throughout the paper. In Section \ref{sec:diagmodel}, we formally define the model and present theorems on diagonalization of linear operators, and present Theorem \ref{thm:ExtUni} proving uniquneness and existence of a solution. In Section \ref{sec:eigfun}, we provide the theory behind diagonalization of operator $\boldsymbol{F}$ under Assumption \ref{assumption:generalF}. 
In Section \ref{sec:existunq}, we prove Theorem \ref{thm:ExtUni}. In Section \ref{sec:continuitybdddiagonal}, we provide theorems on Lipschitz continuity, first in the uniform norm and then using the time-scaled $J_1$ topology.
Supplementary theory and additional proofs are given in Appendix \ref{sec:appendix}.

\subsection{Definitions and Notation} We summarize the commonly used notation here. 
Let $\mathcal{D}_T^k = D([0,T], \mathbb{R}^k)$ be the space of c{\`a}dl{\`a}g functions on $[0,T]$ taking values in $\mathbb{R}^k$. For $x\in\mathcal{D}_T^k$, $x$ has components $x_i(t)$, $t\in[0,T]$, $i=1,\dots, k$. 
Let $\mathcal{C}_T^k$ be the subspace of continuous function in $\mathcal{D}_T^k$. For $k=1$, we simply write $\mathcal{C}_T$ and $\mathcal{D}_T$. 
For $x: [a,b] \to \RR$, we define $\|x\|_1 = \int_{[a,b]} |x(u)|du$

{\it Space  $\mathcal{D}_T(L_1)$.} We say $f \coloneq \{f_u(t):u \in [0,1], t \in [0,T]\} \in \mathcal{D}_T(L_1) $ if $f \in D([0,T], L_1[0,1])$, meaning that (i)
$\|f\|_{T,1}<\infty$, where 
\begin{equation}\label{eq:unifnorm}
\|f\|_{T,1}:= \int_0^1\sup_{0 \leq t \leq T} |f_u(t)|\,du,
\end{equation}
and (ii) for each $u\in[0,1]$, 
$f_u \in \mathcal{D}_T$. 
Note that the dominated convergence theorem guarantees that $[0,T] \ni t \mapsto \{f_u(t) : u\in[0,1]\} \in L_1$ is c{\`a}dl{\`a}g for $f \in \mathcal{D}_T(L_1)$. 
We write $\mathcal{C}_T(L_1)$ to denote the subspace of $\mathcal{D}_T(L_1)$, containing functions that are continuous in $t$, that is, in condition (ii) above,
$f_u \in \mathcal{C}_T$. 
Define $\mathcal{D}_T^\uparrow(L_1) = \{f \in \mathcal{D}_T(L_1): f_u(0) \geq 0, \, d\, f_u(t) \geq 0 \, \forall t\in [0,T], u \in [0,1] \}$, and define $\mathcal{C}^\uparrow_T(L_1)$ analogously. 
For $f,g\in\mathcal{D}_T(L_1)$, we say $f\geq g$ if $f_u(t)\geq g_u(t)$ for all $u\in[0,1]$ and $t\in[0,T]$. 
For $y\in\mathcal{D}_T(L_1)$ and any $t\in[0,T]$, $[f(t)]^+$ is defined for any $u\in[0,1]$ by $[f(t)]^+_u \coloneq [f_u(t)]^+$, where  $[x]^+ \coloneq \max(x,0)$ for $x\in\mathbb{R}$. 

{\it Operators.} For a function $G \colon [0,1]^2 \to \RR_+$, we use the bold symbol to denote the corresponding integral operator $\boldsymbol{G}$ by 
\begin{equation}
\label{eq:Operator}
    (\boldsymbol{G} f)(u) := \int_0^1 G(u,v)f(v)\,dv, \quad f:[0,1] \mapsto\mathbb{R}. 
\end{equation}
Denote the set of integral operators by $\boldsymbol{\mathcal{G}}$. When an operator $\boldsymbol{G}$ is given, it should be understood that the non-bold $G$ represents the corresponding kernel function. In the case where we are dealing with functions in $\mathcal{D}_T(L_1)$, it should be understood that the operator acts on the spatial dimension, not the temporal, and can be written as
\begin{equation}
\label{eq:Operator2}
    (\boldsymbol{G} f)_u(t) := \int_0^1 G(u,v)f_v(t)\,dv,\quad f \in \mathcal{D}_T(L_1).
\end{equation}
Operator $\boldsymbol{G}$ is said to be \textit{nonnegative} and we write $\boldsymbol{G} \ge 0$ if $G(u,v)\geq0$ for all $(u,v) \in [0,1] \times [0,1]$, and $\boldsymbol{G}\le0$ is defined analogously.
The transpose $\boldsymbol{G}^T$ of operator $\boldsymbol{G}$ is defined as
\begin{equation}
\label{eq:OperatorTranspose}
    (\boldsymbol{G}^T f)(u) := \int_0^1 G(v,u)f(v)\,dv\,.
\end{equation}
The $n$-fold composition of operator $\boldsymbol{G}$ is denoted $\boldsymbol{G}^{(n)}$. 
The operator norm $\|\boldsymbol{G}\|_{op}$ of an integral operator is given by
$$\|\boldsymbol{G}\|_{op} \coloneqq \sup_{v \in [0,1]}\int_0^1  |G(u,v)|du. $$

For integral operators $\boldsymbol{F}_1, \boldsymbol{F}_2$ and $f\in\mathcal{D}_T(L_1)$, we have $\|\boldsymbol{F}_1f\|_{T,1}\le \|\boldsymbol{F}_1\|_{op}\|f\|_{T,1}$ and $\|\boldsymbol{F}_1\boldsymbol{F}_2\|_{op}\le\|\boldsymbol{F}_1\|_{op}\|\boldsymbol{F}_2\|_{op}$. These properties are proved in Lemmas A.1 and A.2 of \cite{clarke2026growing}.

The identity operator ${\bf1}$ acts on either $L_1$ or $\mathcal{D}_T(L_1)$ through 
\begin{align*}
(\boldsymbol{1}f)(u) := f(u), \quad f:[0,1] \mapsto\mathbb{R}, \\
(\boldsymbol{1}f)_u(t) := f_u(t), \quad f \in \mathcal{D}_T(L_1).
\end{align*}
As $\|\boldsymbol{1}f\|_{T,1} = \|f\|_{T,1}$, we also define $\|{\bf1}\|_{op}=1$.

For functions $f,g \in \mathcal{D}_T(L_1)$, we define $\operatorname{diag}(f)$ by
$(\operatorname{diag}(f)g)_u(t) = f_u(t)g_u(t)$.

Let $\boldsymbol{F}$ be an operator as defined in \eqref{eq:Operator} and $\rho(\boldsymbol{F})$ be the spectral radius of $\boldsymbol{F}$.  
Define the set of reflection operators: 

\begin{equation} \label{eq:Rdef}
\mathcal{R}:=\{\boldsymbol{F}\in\boldsymbol{\mathcal{G}}: \boldsymbol{F}\ge 0 \,\,\text{and}\,\,  \rho(\boldsymbol{F})<1 \}. 
\end{equation}

Denote
\begin{equation} \label{eq:rtilddef}
\widetilde{\mathcal{R}} := \{\boldsymbol{F}\in\mathcal{R}: \|\boldsymbol{F}\|_{op}<1\}
\end{equation} as the subset of $\mathcal{R}$ with operator norm strictly less than 1. 

For functions $f,g:[0,1]\to\mathbb{R}$, define the diagonal operator $\boldsymbol{D} = \operatorname{diag}(f)$ by
$(\boldsymbol{D}g)(u) = f(u)g(u)$ and $(\boldsymbol{D}^{-1}g)(u) = \frac{g(u)}{f(u)}$. For such an $f$, the operator norm acts as $\|\boldsymbol{D}\|_{op} = \operatorname*{ess\,sup}_{u\in[0,1]}f(u)$ and $\|\boldsymbol{D}^{-1}\|_{op} = (\operatorname*{ess\,inf}_{u\in[0,1]}f(u))^{-1}$, whenever these values are finite. We say diagonal operator $\boldsymbol{D} = \operatorname{diag}(f)$ is strictly positive if $f>0$ a.e. $u\in[0,1]$.

The following property will be used several times throughout the paper: For any two measurable functions $f,g \colon [0,T]\to\RR$ and $t \in [0,T]$, 
\begin{equation}\label{eq:supproperty}  
\left|\sup_{0 \le s \le t} f(s) - \sup_{0 \le s \le t} g(s)\right| \le \sup_{0 \le s \le t} |f(s)-g(s)|.
\end{equation}

\begin{remark}

For processes in $D([0,T],L_1([0,1]))$, an alternative metric could be of the form $\sup_t\|\cdot\|_1$. 
Although this metric is equivalent to the metric $\|\cdot\|_{T,1}$ in the finite-dimensional setting, 
we find it more appropriate to use the metric $\|\cdot\|_{T,1}$ in the infinite-dimensional setting for both technicalities and applications.

Technically, this can be seen in Section \ref{sec:existunq}, where the regulator for the Skorokhod problem is given as the running supremum of $[-X(t)+\boldsymbol{F}Y(t)]^+$. As this is given as the pointwise solution, it makes more natural sense that the norm be of the form $\|\cdot\|_{T,1}$, measuring the $L_1$ norm of the pointwise supremums of each function. For applications, particularly in the mean-field type of analysis of large-scale network models and interacting systems, the convergence of functionals of empirical measures would require the metric $\|\cdot\|_{T,1}$ (see,e.g., \cite{clarke2026growing}). 

\end{remark}

\medskip

\section{An infinite-dimensional Skorokhod problem via diagonalization} \label{sec:diagmodel}

We study  the following system:
\begin{equation} 
    \label{eq:cond1op}
    Z(t) = X(t) +  ({\bf1}-\boldsymbol{F}) Y(t) \geq 0,
\end{equation}
\begin{equation}
    \label{eq:cond2op}
    dY(t) \geq 0 \text{  and  } Y(0) = 0,
\end{equation}
\begin{equation} 
    \label{eq:cond3op}
    Z_u(t)dY_u(t) = 0 \text{ for all } u \in [0,1]\,.
\end{equation}

We say that this set of equations defines an infinite-dimensional Skorokhod mapping in $\mathcal{D}_T(L_1)$. 
We refer to $X$ as the \textbf{free process}. When  we have existence and uniqueness of $(Z,Y)$ given $X$, denote by $\Phi_F(X):=Z$ and $\Psi_F(X):=Y$ the \textbf{reflected process} and \textbf{regulator}, respectively. It should be understood that $dY_u(t)\ge0$ is interpreted as $Y_u(\cdot)$ nondecreasing,  
$dY_u(t)>0$ as $Y_u(\cdot)$ increasing, and that \eqref{eq:cond3op} is interpreted as $Y_u(\cdot)$ increases only when $Z_u(t)=0$.

Given any $X\in\mathcal{D}_T(L_1)$ and $\boldsymbol{F}\in\mathcal{R}$ as defined in \eqref{eq:Rdef}, we aim to prove the following properties:
\begin{itemize}
    \item Existence and uniqueness of $(Y,Z)\in (\mathcal{D}^\uparrow_T(L_1), \mathcal{D}_T(L_1))$ satisfying \eqref{eq:cond1op}--\eqref{eq:cond3op}.
    \item Continuity and Lipschitz of $(Y,Z)$ in both $X$ and $\boldsymbol{F}$.
\end{itemize}

To continue, we first introduce the diagonalized system, in an analogous manner for the finite dimensional Skorokohod mapping in \cite{HarrisonReiman1981reflected}.

\begin{definition}[Diagonalized System]
\label{rem:rephrase}
Let $\boldsymbol{D}$ be a strictly positive diagonal operator.
By letting $\boldsymbol{\widetilde{P}}= \boldsymbol{D}\boldsymbol{F}\boldsymbol{D}^{-1}$, we can continue the analysis of \eqref{eq:cond1op}--\eqref{eq:cond3op}:
\begin{align*}
& Z(t) = X (t)+ ({\bf1}-\boldsymbol{F})Y(t) \\
& \implies \boldsymbol{D}Z(t) = \boldsymbol{D}X(t) + \boldsymbol{D}({\bf1}-\boldsymbol{F})\boldsymbol{D}^{-1}\boldsymbol{D}Y(t) \\
& \implies \boldsymbol{D}Z(t) = \boldsymbol{D}X(t) + ({\bf1}-\boldsymbol{\widetilde{P}})\boldsymbol{D}Y(t)\,.
\end{align*}

In this way, we have replaced $Y$ with $\boldsymbol{D}Y$ and since $\boldsymbol{D}$ is a strictly positive operator, the dynamics of the problem is identical to the original problem.
\\ \\
Further writing $\widetilde{X}=\boldsymbol{D}X$, $\widetilde{Y}=\boldsymbol{D}Y$ and $\widetilde{Z}=\boldsymbol{D}Y$, we can then rewrite \eqref{eq:cond1op}--\eqref{eq:cond3op} as 
\begin{equation} 
\label{eq:cond1opnew}
    \widetilde{Z}(t) = \widetilde{X}(t) +  ({\bf1}-\boldsymbol{\widetilde{P}}) \widetilde{Y}(t) \geq 0,
\end{equation}
\begin{equation}
\label{eq:cond2opnew}
    d\widetilde{Y}(t) \geq 0 \text{  and  } \widetilde{Y}(0) = 0,
\end{equation}
\begin{equation} 
\label{eq:cond3opnew}
    \widetilde{Z}_u(t)d\widetilde{Y}_u(t) = 0 \text{ for all } u \in [0,1].
\end{equation}
We refer to this as the \textbf{diagonalized  system}.
\end{definition}

Note that while the diagonlized system can be defined for any strictly positive diagonal operator, to be of any use for our analysis, we aim to find a diagonal operator $\boldsymbol{D}$ such that the reflection matrix of the diagonalized system $\boldsymbol{\widetilde{P}}$ satisfies $\|\boldsymbol{\widetilde{P}}\|_{op}<1$, i.e., $\boldsymbol{\widetilde{P}}\in\widetilde{\mathcal{R}}$ as defined in \eqref{eq:rtilddef}. We will show that for such a $\boldsymbol{\widetilde{P}}$, in this new system the questions of existence, uniqueness, and Lipschitz continuity can be concisely proved. We provide below an assumption that guarantees this.

\begin{assumption} \label{assumption:generalF}
For $\boldsymbol{F}\in\mathcal{R}$, both of the following properties hold:

\begin{enumerate}[label=\roman*]
    \item[(i)] $\boldsymbol{F}^T$ is strongly positive: for every $f\in L_\infty([0,1])$, $f\ge0$, where $f>0$ on a set of positive measure, there exists $n\ge1$ such that $\operatorname*{ess\,inf}_{u\in[0,1]}((\boldsymbol{F}^T)^{(n)})f(u) > 0$, 
    \item[(ii)] $\boldsymbol{F}^T$ is quasi-compact: $\rho_{ess}(\boldsymbol{F}) < \rho(\boldsymbol{F})$, where $\rho_{ess}(\boldsymbol{F})$ is the essential spectral radius of $\boldsymbol{F}$.
\end{enumerate}
\end{assumption}

We now present a proposition formalizing the discussion above. Its proof is given in Section \ref{sec:eigfun}.

\begin{prop}
\label{prop:OpDiag}
Let $\boldsymbol{F}$ be an operator in $\mathcal{R}$ satisfying Assumption \ref{assumption:generalF}.
Then there exists a diagonal multiplication operator $\boldsymbol{D} = \operatorname{diag}(\delta)$ such that $\operatorname*{ess\,inf}_{u\in[0,1]}\delta(u)>0$, $\operatorname*{ess\,sup}_{u\in[0,1]}\delta(u)<\infty$, and the operator $\boldsymbol{P} \coloneq \boldsymbol{D}^{-1}\boldsymbol{F}^T\boldsymbol{D}$ satisfies
\begin{equation}
\label{eq:diagnorm}
 \int_0^1 P(u,v) \, dv= \rho(\boldsymbol{F}) <1, \quad \forall u\in [0,1].
\end{equation}
Furthermore, letting $\boldsymbol{\widetilde{P}} = \boldsymbol{D}\boldsymbol{F}\boldsymbol{D}^{-1}$, we have 
\begin{equation*}
    \int_0^1 \widetilde{P}(u,v) \, du= \rho(\boldsymbol{F}) <1, \quad \forall u\in [0,1].
\end{equation*}
In particular, $\boldsymbol{\widetilde{P}}\in\widetilde{\mathcal{R}}$, and
\begin{equation*}
    \|\boldsymbol{\widetilde{P}}\|_{op} = \rho(\boldsymbol{F}) < 1\,.
\end{equation*}
\end{prop}

\begin{remark}
This is an infinite-dimensional version of the result for matrices, seen in Theorem 2.5.4 of \cite{bermanplemmons1994}.
\end{remark}

While Assumption \ref{assumption:generalF} provides sufficient conditions to diagonalize the system in a meaningful way, it may be difficult in general to check if an operator $\boldsymbol{F}$ satisfies its conditions.
We present below an interpretable example that satisfies the conditions of Assumption \ref{assumption:generalF}. The proof can be found in 
Appendix \ref{sec:appendix}.

\begin{lemma}
\label{lem:doeblintype} The following condition on $\boldsymbol{F}$ satisfies both conditions (i) and  (ii) of Assumption \ref{assumption:generalF}:

For $\boldsymbol{F}\in\mathcal{R}$, there exists $m\in\NN$, $\epsilon>\frac{1}{2}$, and a probability density function $\nu$ on $[0,1]$ with $\nu>0$ a.e.\ such that for each $u,v \in [0,1]$,
\begin{equation} \label{eq:doeblin}
(F^T)^{(m)}(u,v) \ge \epsilon\nu(v)\,,
\end{equation}
where $(F^T)^{(m)}$ is the kernel associated with operator $(\boldsymbol{F}^T)^{(m)}$.
\end{lemma}

We now present the first major result of the paper.

\begin{theorem} \label{thm:ExtUni}
For any $\boldsymbol{F}\in\mathcal{R}$ satisfying Assumption \ref{assumption:generalF} and any $X \in \mathcal{D}_T(L_1)$ such that $X(0) \geq 0$, there exists  a unique pair $(Y,Z) \in \mathcal{D}^\uparrow_T(L_1)\times\mathcal{D}_T(L_1)$ satisfying \eqref{eq:cond1op}--\eqref{eq:cond3op}.
\end{theorem}
Recall that we denote $\Phi_F(X):=Z$ and $\Psi_F(X):=Y$. The proof for this theorem  is in section \ref{sec:existunq} and will be approached by using at the diagonalized system.

\begin{remark}
The system \eqref{eq:cond1op}--\eqref{eq:cond3op} can be thought of as the following. For any free process $X$, the regulator $Y$ is the least such increasing process such that the resultant process $Z$ remains positive. The term ${\bf1}-\boldsymbol{F}$ acts as an oblique reflection, instead of the typical normal reflection, ensuring that whenever $Z_u(t)=0$, $Y(t)$ will increase accordingly to keep $Z(t)$ from becoming negative. A model of this kind can arise from the fluid limit of an open Markovian Jackson network as the number of servers increases to infinity.
More explicitly, we start with the queue length process $Q^N = \{Q^N_1,\dots,Q^N_N\} \in \mathcal{D}_T^N$ and idle time $I^N \in \mathcal{D}_T^N$, where $I_i(0)=0$, $dI_i(t)\ge0$, and $\int_0^\infty 1_{\{Q^N_i(t)>0\}} dI^N_i(t)=0$ for all $i\in[n]$. If one scales $Q^N$ by $\bar{Q}_i^N(t) := 
N^{-\alpha}Q_i^N(N^\alpha t)$, $\alpha>0$, then one can approximate $\bar{Q}^N$ by a function $\bar{Q}\in\mathcal{C}_T(L_1)$, as we show in \cite{clarke2026growing}. Here, we write $X(t) \coloneq \bar{Q}(0) + \lambda t - ({\bf1}-\boldsymbol{G}^T)\mu t$ with $Y(t) \coloneq \bar{Y}(t) \coloneq \operatorname{diag}(\mu)\bar{I}(t)$ and $Z(t)\coloneq \bar{Q}(t)$, where $\bar{Q}(0)$ represents the initial fluid limit of the queue-length process, $\lambda=(\lambda_u)_{u\in [0,1]}$ and $\mu = (\mu_u)_{u\in [0,1]}$ are the respective scaled arrival and service rate functions in $L_1([0,1])$ for the fluid process, $\boldsymbol{G}$ is the limiting transition operator arising from the routing matrix of the network, and $\bar{I}$ is the fluid limit of the scaled idle time process. For the limiting process it follows that $d\bar{I}_u(t)\geq0$, $\bar{I}_u(0) = 0 $ and $\int_0^\infty 1_{\{\bar{Q}_u(s)>0\}} \bar{I}_u(ds)=0$ for all $u\in[0,1]$. It is easy to verify that these properties satisfy conditions \eqref{eq:cond1op}--\eqref{eq:cond3op}.
\end{remark}

\begin{remark}
One should also note that the finite-dimensional Skorokhod problem may be written in the more general form 
(see, e.g.,  \cite{haddad2010pathwise,ChenYao2001fundamentals})
\begin{equation} \label{eq:skoMmat}
z(t) = x(t) + Ry(t),
\end{equation}
(see, e.g., \cite{haddad2010pathwise,ChenYao2001fundamentals})
where $R$ is a reflection matrix. It is assumed that $R$ is in the space of M-matrices (positive diagonal elements, non
positive off-diagonal elements, and a nonnegative inverse). Given such an $R$, one can write $R=(I-F)D$, where $D$ is a diagonal matrix with diagonal entries equal to those of $R$ and $F$ is a positive matrix with $\rho(F)<1$. Conversely, given a positive matrix $F$ with $\rho(F)<1$, it follows that $R$ is an M-matrix. In this way, solving the finite-dimensional Skorokhod problem the form \eqref{eq:skoMmat} is equivalent to solving the problem in the form \eqref{eq:origskorokhod}.
This relationship applies to our infinite-dimensional setting. We can instead define the Skorokhod problem via the equation
\begin{equation} \label{eq:skoMop}
Z(t) = X(t) + \boldsymbol{R}Y(t).
\end{equation}
Here, $\boldsymbol{R}\in\bar{\mathcal{R}}$, where $\bar{\mathcal{R}}\subset\boldsymbol{\mathcal{G}}$ is defined as the set of operators that can be written in the form $\boldsymbol{R} = ({\bf1}-\boldsymbol{F})\boldsymbol{D}$, where $\boldsymbol{F}\in\mathcal{R}$ and $\boldsymbol{D}$ is a strictly positive diagonal operator, analogous to the set of M-matrices in finite-dimension. Clearly by letting $\widetilde{Y}=\boldsymbol{D}Y$, we recover the system determined by \eqref{eq:cond1op}--\eqref{eq:cond3op} with $Y$ replaced by $\widetilde{Y}$.

One may also note that M-matrices can be equivalently defined as $R = sI - B$, where $B\ge0$ element-wise and $\rho(B)<s$. We can similarly define $\bar{\mathcal{R}}$ through $\boldsymbol{R} = s{\bf1}-\boldsymbol{B}$, where $\rho(\boldsymbol{B})<1$. In this case, letting $\boldsymbol{F} = \frac{1}{s}\boldsymbol{B}$ and writing $\widetilde{Y} = sY$ we see that in this form, \eqref{eq:skoMmat} is still equivalent to \eqref{eq:cond1op}.

\end{remark}

\bigskip

\section{Eigenfunctions of operators in $\mathcal{R}$ and Proof of Proposition \ref{prop:OpDiag}}
\label{sec:eigfun}

In this section we establish the necessary theory to ensure that the diagonalization argument established in Definition \ref{rem:rephrase} can be applied in a reconstructible manner; that is, for any results we achieve under the diagonalized system, we can revert to the initial system. In order to do this in a meaningful way one would hope that the diagonal operators are strictly bounded away from $0$ and infinity. We will show in this section that indeed Assumption \ref{assumption:generalF} achieves such conditions, by providing background theory and a proof to Proposition \ref{prop:OpDiag}. We begin with a generalization of the Krein-Rutman theorem that applies directly to the conditions of Assumption \ref{assumption:generalF}.

\begin{lemma}
\label{lem:kreinrutman}
Let $\mathcal{X} = L_{\infty}([0,1])$, $\mathcal{K} = \{f\in X\colon f\ge0 \mbox{ a.e}. \}$, and let $\mathcal{T} : \mathcal{X} \rightarrow \mathcal{X}$ be a bounded linear operator such that 
\begin{itemize}
    \item $\mathcal{T}$ is positive: $\mathcal{T}(\mathcal{K}) \subset \mathcal{K}$,
    \item $\mathcal{T}$ is quasi-compact: $\rho_{ess}(\mathcal{T}) < \rho(\mathcal{T})$, where $\rho_{ess}(\mathcal{T})$ is the essential spectral radius of $\mathcal{T}$,
    \item $\mathcal{T}$ is eventually strongly positive: for every $f\in \mathcal{K}\backslash \{0\}$, there exists $n\ge1$ such that $\mathcal{T}^{(n)}f\in \mbox{int}(\mathcal{K})$, where $\mbox{int}(\mathcal{K}) = \{f\in X\colon \operatorname*{ess\,inf} f>0\}$. 
\end{itemize}
Then $\rho(\mathcal{T})$ is an eigenvalue of $\mathcal{T}$ and there exists a unique (up to scaling) eigenfunction $\upsilon \in \mbox{int}(\mathcal{K})$ such that $\mathcal{T}\upsilon = \rho(\mathcal{T})\upsilon$. Furthermore, $\upsilon$ is strictly bounded above $0$ in the sense that $ \operatorname*{ess\,inf}_{u\in[0,1]}\upsilon(u) >0$.
\end{lemma}

\begin{proof}
First, applying \cite[Corollary 2.2]{nussbaum1980kreinrutman}, as $\mathcal{T}$ is quasi-compact,  we are guaranteed that there exists an eigenfunction $\upsilon$ such that $\mathcal{T}\upsilon = \rho(\mathcal{T})\upsilon$ and $\upsilon \ge 0$ a.e.  Then, by eventual strong positivity,
$$\operatorname*{ess\,inf}_{u\in[0,1]}\,(\mathcal{T}^{(n)}\upsilon)(u) >0 \implies \rho(\mathcal{T})^n\operatorname*{ess\,inf}_{u\in[0,1]}\upsilon(u) >0\,,$$
and so $\operatorname*{ess\,inf}_{u\in[0,1]}\upsilon(u) > 0$ a.e. as required.
\end{proof}

Note that by definition, $\boldsymbol{F}\in\mathcal{R}$ acts on $L_1([0,1])$ and its transpose (or adjoint) $\boldsymbol{F}^T$ acts on $L_{\infty}([0,1])$.
It is then clear to see that for any $\boldsymbol{F}\in\mathcal{R}$ satisfying Assumption \ref{assumption:generalF}, its transpose $\boldsymbol{F}^T$ satisfies the conditions of Lemma \ref{lem:kreinrutman}.

\begin{lemma}
\label{lem:diagdef}
Let $\boldsymbol{F}$ be an integral operator and $\boldsymbol{D} = \operatorname{diag}(\delta)$ be a diagonal operator.
Then the operator $\boldsymbol{P} \coloneq \boldsymbol{D}^{-1}\boldsymbol{F}\boldsymbol{D}$ satisfies 
\begin{equation*}
(\boldsymbol{P}f)(u) = \int_0^1 P(u,v)f(v) \, dv\,,
 \end{equation*}
where
\begin{equation*}
 P(u,v) \coloneq \frac{F(u,v)\delta(v)}{\delta(u)}.
 \end{equation*}
\end{lemma}

\begin{proof}
Using the properties of linear and diagonal operators, we obtain
\begin{align*}
(\boldsymbol{P}f)(u) = (\boldsymbol{D}^{-1}\boldsymbol{F}\boldsymbol{D}f)(u) =  \frac{1}{\delta(u)} \int_0^1 F(u,v) \delta(v) f(v) \, dv = \int_0^1 P(u,v)f(v) \, dv.
\end{align*}
\end{proof}

We now have the necessary results to prove Proposition \ref{prop:OpDiag}. 

\begin{proof}[Proof of Proposition \ref{prop:OpDiag}]
From Lemma \ref{lem:kreinrutman} we know that $\boldsymbol{F}^T$ has an eigenfunction $\delta$ such that $\operatorname*{ess\,inf}_{u\in[0,1]}\delta(u)>0$ and $\boldsymbol{F}^T\delta = \rho(\boldsymbol{F}^T)\delta=\rho(\boldsymbol{F})\delta$. Since $\boldsymbol{F}^T$ acts on $L_\infty([0,1])$ we also know that $\operatorname*{ess\,sup}_{u\in[0,1]}\delta(u)<\infty$.
By Lemma \ref{lem:diagdef} applied to $\boldsymbol{F}^T$, the kernel function associated with $\boldsymbol{P}$ satisfies $P(u,v) = \frac{F(v,u)\delta(v)}{\delta(u)}$. Putting this together, it follows that for any $u\in[0,1]$,
$$ \int_0^1 P(u,v) \, dv =  \frac{1}{\delta(u)} \int_0^1 F(v,u) \delta(v) \, dv = \frac{1}{\delta(u)} (\boldsymbol{F}^T \delta)(u) = \frac{1}{\delta(u)} \rho(\boldsymbol{F}) \delta(u)  = \rho(\boldsymbol{F})<1\,.$$
To prove the bound on the norm of $\boldsymbol{\widetilde{P}}$, we show that $\boldsymbol{\widetilde{P}} = \boldsymbol{P}^T$. Indeed, it is clear that $\boldsymbol{\widetilde{P}}$ has kernel $\widetilde{P}(u,v) = \frac{F(u,v)\delta(u)}{\delta(v)}$ and so
$$(\boldsymbol{\widetilde{P}}f)(u) = \int_0^1\frac{F(u,v)\delta(u)}{\delta(v)}f(v)\,dv = \int_0^1P(v,u)f(v)\,dv = (\boldsymbol{P}^Tf)(u).$$
Using \eqref{eq:diagnorm} we see that
    \begin{align*}
        &\int_0^1 \widetilde{P}(u,v) \, du = \int_0^1 P(u,v) \, dv = \rho(\boldsymbol{F})<1. 
    \end{align*}
    Finally, note that since $\boldsymbol{\widetilde{P}}$ is positive,
    \begin{align*}
    \|\boldsymbol{\widetilde{P}}\|_{op} = \sup_{v \in [0,1]}\int_0^1|\widetilde{P}(u,v)|du = \sup_{v \in [0,1]}\int_0^1\widetilde{P}(u,v)du = \rho(\boldsymbol{F}). 
    \end{align*}
    This completes the proof.
\end{proof}

\section{Proof of Existence and Uniqueness of the Reflection Map} \label{sec:existunq}

In this section, we will introduce necessary results to prove Theorem \ref{thm:ExtUni}, utilizing the property that the diagonalized system has an operator $\boldsymbol{\widetilde{P}}$ such that $\|\boldsymbol{\widetilde{P}}\|_{op}<1$. This avoids any work involving $k$-fold composition that was required in the methodology of \cite{clarke2026growing}, for which we only knew that $\|\boldsymbol{F}^{(k)}\|_{op}<1$ for some $k\ge1$. 
For ease of notation in this section, the lemmas are proved for $\boldsymbol{F}\in\widetilde{\mathcal{R}}$, although it should be understood that for the proof of Theorem \ref{thm:ExtUni} these lemmas will apply to the operator $\boldsymbol{\widetilde{P}}\in\widetilde{\mathcal{R}}$ of the diagonalized system.

\bigskip

Before proving Theorem \ref{thm:ExtUni} we first introduce the mapping $\pi_{X,F}$ which acts as a cumulative correction term for candidate regulators $Y\in\mathcal{D}^\uparrow_T(L_1)$. 

\begin{definition}
For a given $X\in \mathcal{D}_T(L_1)$ and $\boldsymbol{F}\in\mathcal{R}$, define the mapping $\pi_{X,F}:\mathcal{D}^\uparrow_T(L_1)\to\mathcal{D}^\uparrow_T(L_1)$ by
\begin{equation} \label{eq:pidef}
\pi_{X,F}(W)(t) = \sup_{0\leq s \leq t}[-X(s) + \boldsymbol{F}W(s)]^+\,.
\end{equation} 
More specifically, for each $u \in [0,1]$,
\begin{equation} \label{eq:pidefu}
\pi_{X,F}(W)_u(t) = \sup_{0\leq s \leq t}[-X_u(s) + (\boldsymbol{F}W)_u(s)]^+\,.
\end{equation}
\end{definition}

\begin{lemma} \label{lem:picontraction}
For any $\boldsymbol{F}\in\widetilde{\mathcal{R}}$ and $X\in\mathcal{D}_T(L_1)$, $\pi_{X,F}$ is a contraction mapping on $\mathcal{D}^\uparrow_T(L_1)$. That is, for any $Y^1,Y^2 \in \mathcal{D}^\uparrow_T$, 
$$\|\pi_{X,F}(Y^1)-\pi_{X,F}(Y^2)\|_{T,1} \leq \|\boldsymbol{F}\|_{op}\|Y^1-Y^2\|_{T,1}$$
where $\|\boldsymbol{F}\|_{op}<1$. 
\end{lemma}

\begin{proof}
The proof is given as follows:
\begin{align*}
\|\pi_{X,F}(Y^1)-\pi_{X,F}(Y^2)\|_{T,1} & = \int_0^1\sup_{t\in[0,T]}|\pi_{X,F}(Y^1)_u(t)-\pi_{X,F}(Y^2)_u(t)|du\\
&\leq \int_0^1\sup_{t\in[0,T]}\sup_{s\in[0,t]}|(\boldsymbol{F}(Y^1-Y^2))_u(s)|du \\
&= \int_0^1\sup_{t\in[0,T]}|(\boldsymbol{F}(Y^1-Y^2))_u(t)|du\\
&=\|\boldsymbol{F}(Y^1)-\boldsymbol{F}(Y^2)\|_{T,1}\\
&\leq \|\boldsymbol{F}\|_{op}\|Y^1-Y^2\|_{T,1}\,,
\end{align*}
where we use \eqref{eq:supproperty} in the second line.
\end{proof}

\begin{remark}
In the finite-dimensional setting, the above result can be proved similarly, see \cite[Theorem 1]{HarrisonReiman1981reflected}. In finite-dimension, this is sufficient to show that a unique fixed point exists where uniqueness means for every $t\in[0,T]$ and every $i=1,\dots N$, where the function space is $\mathcal{D}_T^N$. This will not work here since in our case, uniqueness in the function space $\mathcal{D}_T(L_1)$ under the metric $\|\cdot\|_{T,1}$  does not imply uniqueness for all $t\in[0,T]$, but only uniqueness for almost every $u\in[0,1]$. This is not sufficient to prove uniqueness of the system \eqref{eq:cond1op}--\eqref{eq:cond3op}, which holds for all $u\in[0,1]$. We therefore require a result which improves upon almost everywhere uniqueness to everywhere. This is presented in the next lemma.
\end{remark}

\begin{lemma} \label{lem:pifixpoint}
For any $\boldsymbol{F}\in\widetilde{\mathcal{R}}$ and $X\in\mathcal{D}_T(L_1)$, $\pi_{X,F}$ has a unique fixed point. That is, there exists a unique $Y\in\mathcal{D}_T(L_1)$ such that $Y = \pi_{X,F}(Y)$, where uniqueness holds for every $u\in[0,1]$ and $t\in[0,T]$.
\end{lemma}

\begin{proof}
Applying the Banach fixed point theorem to Lemma \ref{lem:picontraction}, there exists a unique $Y\in\mathcal{D}_T(L_1)$ with respect to the norm $\|\cdot\|_{T,1}$ such that $Y_u = \pi_{X,F}(Y)_u$ for a.e.\ $u\in[0,1]$.
Let $\widetilde{Y} := \pi_{X,F}(Y)$. 
Then $\widetilde{Y}_u = Y_u$ for a.e.\ $u\in[0,1]$.
Hence, $\pi_{X,F}(\widetilde{Y})_u = \pi_{X,F}(Y)_u$ for every $u\in[0,1]$.
Therefore, $\widetilde{Y} = \pi_{X,F}(\widetilde{Y})$, namely $\widetilde{Y}$ is a fixed point of $\pi_{X,F}$.

Now suppose $\widehat{Y}$ is also a fixed point of $\pi_{X,F}$.
By uniqueness of $Y$, we must have $\widehat{Y}_u = Y_u$ for a.e.\ $u\in[0,1]$.
Hence, $\pi_{X,F}(\widehat{Y})_u = \pi_{X,F}(Y)_u$ for every $u\in[0,1]$.
Therefore, $\widehat{Y} = \pi_{X,F}(\widehat{Y}) = \pi_{X,F}(Y) = \widetilde{Y}$.
This means $\widetilde{Y}$ is the unique fixed point of $\pi_{X,F}$.
\end{proof}

We now have the necessary results to prove existence and uniqueness of the solution to the infinite-dimensional Skorokhod problem. The following proof is adapted from \cite{HarrisonReiman1981reflected} and further developed for the infinite-dimensional case.

\begin{proof}[Proof of Theorem \ref{thm:ExtUni}]
By using the diagonalized system from Definition \ref{rem:rephrase} where $\boldsymbol{\widetilde{P}}\in\widetilde{\mathcal{R}}$ as guaranteed by Proposition \ref{prop:OpDiag}, we can consider the system \eqref{eq:cond1opnew}--\eqref{eq:cond3opnew}.
By proving existence and uniqueness of $\widetilde{Y}$ and $\widetilde{Z}$ for a given free process $\widetilde{X} \coloneq \boldsymbol{D}X$ and operator $\boldsymbol{\widetilde{P}}$, we can recover the original theorem taking $Y \coloneq \boldsymbol{D}^{-1}\widetilde{Y}$ and $Z \coloneq \boldsymbol{D}^{-1}\widetilde{Z}$.  
\\ \\
We first prove existence by showing that a solution to \eqref{eq:cond1opnew}--\eqref{eq:cond3opnew} is given by the unique fixed point $\widetilde{Y} = \pi_{\widetilde{X},\widetilde{P}}(\widetilde{Y})$ guaranteed by Lemma \ref{lem:pifixpoint} and then correspondingly, $\widetilde{Z}(t) = \widetilde{X}(t) +  ({\bf1}-\boldsymbol{\widetilde{P}}) \widetilde{Y}(t)$.

Let $\widetilde{Y}$ be the fixed point. 
Condition \eqref{eq:cond1opnew} is satisfied since 
\begin{align*}
\widetilde{Z}(t)
& = \widetilde{X}(t)-\boldsymbol{\widetilde{P}}\widetilde{Y}(t)+\widetilde{Y}(t) \\
&= \widetilde{X}(t)-\boldsymbol{\widetilde{P}}\widetilde{Y}(t)+\sup_{0\leq s \leq t}[-\widetilde{X}(s) + \boldsymbol{\widetilde{P}}\widetilde{Y}(s)]^+ \geq 0\,. 
\end{align*}

Condition \eqref{eq:cond2opnew} is also satisfied since clearly $d\pi_{\widetilde{X},\widetilde{P}}(\widetilde{Y})(t) \geq 0$ and $\widetilde{Y}(0) = [-\widetilde{X}(0) + \boldsymbol{\widetilde{P}}\widetilde{Y}(0)]^+$ is satisfied when $\widetilde{Y}(0) = 0$. Since $\pi_{\widetilde{X},\widetilde{P}} \colon \mathcal{D}^\uparrow_0(L_1) \to \mathcal{D}^\uparrow_0(L_1)$ is a contraction map (taking $T=0$ in Lemma \ref{lem:picontraction}), the uniqueness of $\widetilde{Y}(0)$ guarantees that $\widetilde{Y}(0)=0$. 

Finally, to see \eqref{eq:cond3opnew}, we can check both cases. For each $u \in [0,1]$, if $d\widetilde{Y}_u(t) > 0$, the supremum in $\pi(\widetilde{Y})_u(t)$ must be attained at $t$ and so $\widetilde{Z}_u(t) = \widetilde{X}_u(t)-(\boldsymbol{\widetilde{P}}\widetilde{Y})_u(t) + [-\widetilde{X}_u(t)+(\boldsymbol{\widetilde{P}}\widetilde{Y})_u(t)] = 0$. Alternatively, if $\widetilde{Z}_u(t) > 0$, then $\widetilde{X}(t)-\boldsymbol{\widetilde{P}}\widetilde{Y}(t)+\sup_{0\leq s \leq t}[-\widetilde{X}(s) + \boldsymbol{\widetilde{P}}\widetilde{Y}(s)]^+>0$ and so the supremum in $\pi_{\widetilde{X},\widetilde{P}}(\widetilde{Y})_u(t)$ can not be attained at $t$; therefore $d\widetilde{Y}_u(t) = 0$.

We next show that the fixed point $\widetilde{Y} = \pi_{\widetilde{X},\widetilde{P}}(\widetilde{Y})$ and $\widetilde{Z}(t) = \widetilde{X}(t) +  ({\bf1}-\boldsymbol{\widetilde{P}}) \widetilde{Y}(t)$ is the only solution to \eqref{eq:cond1opnew}--\eqref{eq:cond3opnew}, 
proving uniqueness. To see this, first take any $(\bar{Y},\bar{Z})$ that satisfies \eqref{eq:cond1opnew}--\eqref{eq:cond3opnew}, and write $\hat{Y} = \pi_{\widetilde{X},\widetilde{P}}(\bar{Y})$. We will show that $\hat{Y} = \bar{Y}$. Note first that $\hat{Y} \le \bar{Y}$. To see this, note that conditions \eqref{eq:cond1opnew} and \eqref{eq:cond2opnew} imply that for all $u\in[0,1], t\in[0,T]$, 
$$\bar{Y}_u(t) \geq [-\widetilde{X}_u(t) + (\boldsymbol{\widetilde{P}}\bar{Y})_u(t)]^+\,.$$
For any fixed $t$ and $\epsilon>0$, $\exists\, t^*\in[0,t]$ such that 
$$[-\widetilde{X}_u(t^*) + (\boldsymbol{\widetilde{P}}\bar{Y})_u(t^*)]^+ + \epsilon \geq \sup_{0\leq s\leq t}[-\widetilde{X}_u(s) + (\boldsymbol{\widetilde{P}}\bar{Y})_u(s))]^+\,.$$
Since $d\bar{Y}(t) \geq 0$, we have 
\begin{align*}
\bar{Y}_u(t) \geq \bar{Y}_u(t^*) 
& \geq [-\widetilde{X}_u(t^*) + (\boldsymbol{\widetilde{P}}\bar{Y})_u(t^*)]^+ \\& \geq \sup_{0\leq s\leq t}[-\widetilde{X}_u(s) + (\boldsymbol{\widetilde{P}}\bar{Y})_u(s))]^+ - \epsilon = \hat{Y}_u(t) - \epsilon\,. 
\end{align*}
Letting $\epsilon \rightarrow 0$ yields $\hat{Y}_u(t) \leq \bar{Y}_u(t)$.

We next show that $\hat{Y} = \bar{Y}$, implying uniqueness of $\bar{Y}$. If not, since $\bar{Y}(0) = 0 = \hat{Y}(0)$, there must exist $u\in [0,1], t_0\in[0,T]$ such that both $\hat{Y}_u(t_0) < \bar{Y}_u(t_0)$ and $d\bar{Y}_u(t_0) > 0$. 
However, this means that $\bar{Z}_u(t_0) \geq \bar{Y}_u(t_0) - \hat{Y}_u(t_0) > 0$, contradicting \eqref{eq:cond3opnew}.
\end{proof}

\begin{remark}
In this section we have shown that the regulator for the diagonalized system \eqref{eq:cond1opnew}--\eqref{eq:cond3opnew} is given by the unique fixed point of $\pi_{\widetilde{X},\widetilde{P}}$. As it turns out, one can also show (cf.\  \cite{clarke2026growing} ) that the regulator of the original system \eqref{eq:cond1op}--\eqref{eq:cond3op} is given by the unique fixed point of $\pi_{X,F}$, but this requires additional theory and is less straightforward than the results presented in this section. Specifically, when $\boldsymbol{F}$ is in $\mathcal{R}$ but not necessarily $\widetilde{\mathcal{R}}$, contraction properties of $\pi_{X,F}$ can only be shown via a $k$-stage contraction argument, where $k$ is a positive integer such that $\|\boldsymbol{F}^{(k)}\|_{op}<1$. Additionally, for the Lipschitz properties that follow in Section \ref{sec:continuitybdddiagonal}, if $\|\boldsymbol{F}\|_{op}=1$ the proofs do not hold and so more delicate proofs bounding compositions of $\boldsymbol{F}$ and $\pi_{X,F}$ must be taken.
\end{remark}

\medskip

\section{Continuity and Lipschitz Properties}
\label{sec:continuitybdddiagonal}

In the previous section we have demonstrated that given any reflection operator $\boldsymbol{F}\in\mathcal{R}$ satisfying Assumption \ref{assumption:generalF} and free process $X\in\mathcal{D}_T(L_1)$, we can always find a unique solution to the infinite-dimensional Skorokhod problem posed by \eqref{eq:cond1op}--\eqref{eq:cond3op}. In this section we will show that under the same assumptions on $\boldsymbol{F}$, both the regulator $\Psi_F(X)$ and the reflected process $\Phi_F(X)$ are Lipschitz in both $X$ and $\boldsymbol{F}$. In a similar fashion to the proofs in the previous section, we will again utilize the diagonalized system, before reverting back to the original system. It is important to remember that when we show these proofs for $\boldsymbol{F}\in\widetilde{\mathcal{R}}$, we are referring to processes in the diagonalized system where the reflection operator has the property that $\|\boldsymbol{F}\|_{op}<1$.

In Section \ref{sec:eigfun} we have shown that Assumption \ref{assumption:generalF} allows for diagonalization of the system with a dominant eigenfunction whose essential infimum is bounded above zero. In particular, for the diagonalized system $\boldsymbol{\widetilde{P}}= \boldsymbol{D}\boldsymbol{F}\boldsymbol{D}^{-1}$, we have that $\|\boldsymbol{\widetilde{P}}\|_{op} = \rho(\boldsymbol{F})<1$. Furthermore, Lemma \ref{lem:kreinrutman} states that 
$$0 < \frac{\|\boldsymbol{D}\|_{op}}{\|\boldsymbol{D}^{-1}\|_{op}} < \infty,$$
which allows for the return to the original system without any issues regarding boundedness. This will be crucial in proving the Lipschitz properties in this section.

\subsection{Lipschitz continuity in the uniform norm} \label{subsec:uniflip}
We prove results on Lipschitz continuity in the uniform norm $\|\cdot\|_{T,1}$ established in \eqref{eq:unifnorm} in this subsection.

In Lemma \ref{lem:pifixpoint} we have shown that for $\boldsymbol{F}\in\widetilde{\mathcal{R}}$, $\pi_{X,F}$ has a unique fixed point. In the next lemma we show that we can reach this fixed point by iterating $\pi_{X,F}$ on any input process $W\in\mathcal{D}^\uparrow_T(L_1)$. 

\begin{lemma}
\label{lem:pitoreg}
Take $\boldsymbol{F}\in\widetilde{\mathcal{R}}$ and $X\in\mathcal{D}_T(L_1)$. Then for any $W\in\mathcal{D}_T^\uparrow(L_1)$, \[\|\pi_{X,F}^n(W) - \Psi_F(X)\|_{T,1}\to0 \quad\text{ as} n\to\infty.\]
\end{lemma}

\begin{proof}
Take $W^1,W^2\in\mathcal{D}^\uparrow_T(L_1)$. Applying Lemma \ref{lem:picontraction} iteratively gives
$$\|\pi^n_{X,F}(W^1) - \pi^n_{X,F}(W^2)\|_{T,1} \le \|\boldsymbol{F}\|_{op}^n\|W^1-W^2\|_{T,1}.$$
Next, let $W^2$ be the unique fixed point of $\pi_{X,F}$ guaranteed by Lemma \ref{lem:pifixpoint}, so $W^2=\pi_{X,F}(W^2)=\Psi_{F}(X)$. As $\|\boldsymbol{F}\|_{op}<1$, the result follows immediately.
\end{proof}

In the next theorem we show that given a fixed reflection operator $\boldsymbol{F}\in\mathcal{R}$,
$\Psi_F(X)$ and $\Phi_F(X)$ are both Lipschitz in $X$. 

\begin{theorem}\label{thm:DiagLipschitz}
Take $X^1\in\mathcal{D}_T(L_1)$ and $X^2\in\mathcal{D}_T(L_1)$. Let $(\Phi_{F}(X^1),\Psi_{F}(X^1))$ and $(\Phi_{F}(X^2),\Psi_{F}(X^2))$ be the corresponding reflected processes and regulators associated with operator $\boldsymbol{F}\in\mathcal{R}$ satisfying Assumption \ref{assumption:generalF}. Let $\boldsymbol{D}$ be a strictly positive diagonal operator such that $\|\boldsymbol{D}\boldsymbol{F}\boldsymbol{D}^{-1}\|_{op} = \gamma<1$.
Then the regulator and reflected process maps of the original process are Lipschitz in the free process, satisfying 
    \begin{equation*}
    \|\Psi_F(X^1)-\Psi_F(X^2)\|_{T,1} \le \frac{\delta_{\sup}}{\delta_{\inf}(1-\gamma)} \|X^1-X^2\|_{T,1}\,,
    \end{equation*}
    \begin{equation*}
        \|\Phi_F(X^1)-\Phi_F(X^2)\|_{T,1} \le \frac{\delta_{\sup}}{\delta_{\inf}}\left(1 + \frac{2}{1-\gamma}\right)\|X^1 - X^2\|_{T,1},
    \end{equation*}
    where $\delta_{\sup} = \|\boldsymbol{D}\|_{op}$ and $\delta_{\inf}^{-1} = \|\boldsymbol{D}^{-1}\|_{op}$.
\end{theorem}

\begin{proof}
Consider the diagonalized processes $\widetilde{X}^i \coloneq \boldsymbol{D}X^i$, for $i=1,2$, with operator $\boldsymbol{\widetilde{P}} = \boldsymbol{D}\boldsymbol{F}\boldsymbol{D}^{-1}$ where $\boldsymbol{\widetilde{P}}\in\widetilde{\mathcal{R}}$ and $\|\boldsymbol{\widetilde{P}}\|_{op} = \gamma <1$. Denote $\Phi_{\widetilde{P}}(\widetilde{X}^i) = \widetilde{Z}^i$ and $\Psi_{\widetilde{P}}(\widetilde{X}^i) = \widetilde{Y}^i$ the reflected process and the regulator corresponding to $\boldsymbol{\widetilde{P}}$, respectively.

Let $\widetilde{Y}^{i,0} = 0$ and $\widetilde{Y}^{i,n+1} = \pi_{\widetilde{X}^i,\widetilde{P}}(\widetilde{Y}^{i,n})$ for $n \geq 0, i = 1,2$. Then, using \eqref{eq:supproperty} in the first inequality, we obtain
\begin{align*}
&\|\widetilde{Y}^{1,n+1}-\widetilde{Y}^{2,n+1}\|_{T,1}\\
&=\|\pi_1(\widetilde{Y}^{1,n}) - \pi_2(\widetilde{Y}^{2,n})\|_{T,1} \\
&=\int_0^1\sup_{0\leq t\leq T}\left|\sup_{0\leq s \leq t}[-\widetilde{X}_u^1(s) + (\boldsymbol{\widetilde{P}}\widetilde{Y}^{1,n})_u(s)]^+ - \sup_{0\leq s \leq t}[-\widetilde{X}_u^2(s) + (\boldsymbol{\widetilde{P}}\widetilde{Y}^{2,n})_u(s)]^+\right|du\\
&\leq \int_0^1\sup_{0\leq t\leq T}\left|[-\widetilde{X}_u^1(t) + (\boldsymbol{\widetilde{P}}\widetilde{Y}^{1,n})_u(t)]^+ - [-\widetilde{X}_u^2(t) + (\boldsymbol{\widetilde{P}}\widetilde{Y}^{2,n})_u(t)]^+\right|du\\
&\leq \int_0^1\sup_{0\leq t\leq T}\left|-(\widetilde{X}_u^1(t)-\widetilde{X}_u^2(t)) + ((\boldsymbol{\widetilde{P}}\widetilde{Y}^{1,n})_u(t) -(\boldsymbol{\widetilde{P}}\widetilde{Y}^{2,n})_u(t))\right|du\\
&\leq \int_0^1\sup_{0\leq t\leq T}\left|\widetilde{X}_u^1(t)-\widetilde{X}_u^2(t)\right|du + \int_0^1\sup_{0\leq t\leq T}\left|\boldsymbol{\widetilde{P}}(\widetilde{Y}^{1,n}-\widetilde{Y}^{2,n})_u(t)\right|du\\\\
&\leq \|\widetilde{X}^1-\widetilde{X}^2\|_{T,1} + \gamma\|\widetilde{Y}^{1,n}-\widetilde{Y}^{2,n}\|_{T,1}
\end{align*}
where the last inequality follows from the proof of Lemma \ref{lem:picontraction}.

Applying the reverse triangle inequality, we get
$$\big|\|\widetilde{Y}^{1,n}-\widetilde{Y}^{2,n}\|_{T,1}-\|\Psi_{\widetilde{P}}(\widetilde{X}^1)-\Psi_{\widetilde{P}}(\widetilde{X}^2)\|_{T,1}\big| \le \|\widetilde{Y}^{1,n}-\Psi_{\widetilde{P}}(\widetilde{X}^1)\|_{T,1}+\|\Psi_{\widetilde{P}}(\widetilde{X}^2)-\widetilde{Y}^{2,n}\|_{T,1},$$ and so applying Lemma \ref{lem:pitoreg} yields \[\|\widetilde{Y}^{1,n}-\widetilde{Y}^{2,n}\|_{T,1} \to \|\Psi_{\widetilde{P}}(\widetilde{X}^1)-\Psi_{\widetilde{P}}(\widetilde{X}^2)\|_{T,1}.\] 
Letting $n \rightarrow \infty$ and rearranging, we get 
$$\|\Psi_{\widetilde{P}}({\widetilde{X}}^1)-\Psi_{\widetilde{P}}({\widetilde{X}}^2)\|_{T,1} \le \frac{1}{1-\gamma} \|{\widetilde{X}}^1-{\widetilde{X}}^2\|_{T,1}\,. $$
Using \eqref{eq:cond1opnew} and the fact that $\|\mathbf{1}-\boldsymbol{\widetilde{P}}\|_{op} \le 2$, we obtain 
\begin{align*}
\|\Phi_{\widetilde{P}}(\widetilde{X}^1)-\Phi_{\widetilde{P}}(\widetilde{X}^2)\|_{T,1} &\leq \|\widetilde{X}^1-\widetilde{X}^2\|_{T,1} + \|\mathbf{1}-\boldsymbol{\widetilde{P}}\|_{op}\|\Psi_{\widetilde{P}}(\widetilde{X}^1)-\Psi_{\widetilde{P}}(\widetilde{X}^2)\|_{T,1} \\
&\leq \left(1 + \frac{2}{1-\gamma}\right)\|\widetilde{X}^1 - \widetilde{X}^2\|_{T,1}\,.
\end{align*}

For the corresponding $X^1$ and $X^2$ relating back to the original operator $F$, we have the relation

$$\Psi_F(X^i) = Y^i = \boldsymbol{D}^{-1}\widetilde{Y}^i = \boldsymbol{D}^{-1}\Psi_{\widetilde{P}}(\widetilde{X}^i)= \boldsymbol{D}^{-1}\Psi_{\widetilde{P}}(\boldsymbol{D}X^i), \quad i=1,2\,,$$
and similarly $\Phi_F(X^i) = \boldsymbol{D}^{-1}\Phi_{\widetilde{P}}(\boldsymbol{D}X^i)\,.$

Then we obtain 
\begin{align*}
\|\Psi_F(X^1)-\Psi_F(X^2)\|_{T,1} &= \|\boldsymbol{D}^{-1}\Psi_{\widetilde{P}}(\boldsymbol{D}X^1)-\boldsymbol{D}^{-1}\Psi_{\widetilde{P}}(\boldsymbol{D}X^2)\|_{T,1} \\
&\leq \frac{1}{\delta_{\inf}}\|\Psi_{\widetilde{P}}(\boldsymbol{D}X^1)-\Psi_{\widetilde{P}}(\boldsymbol{D}X^2)\|_{T,1} \\
&\leq \frac{1}{\delta_{\inf}(1-\gamma)} \|\boldsymbol{D}X^1-\boldsymbol{D}X^2\|_{T,1} \\
&\leq\frac{\delta_{\sup}}{\delta_{\inf}(1-\gamma)} \|X^1-X^2\|_{T,1}\,,
\end{align*}
and similarly,
\begin{align*}
\|\Phi_F(X^1)-\Phi_F(X^2)\|_{T,1} &\leq \frac{1}{\delta_{\inf}}\|\Phi_{\widetilde{P}}(\boldsymbol{D}X^1)-\Phi_{\widetilde{P}}(\boldsymbol{D}X^2)\|_{T,1} \\
&\leq \frac{\delta_{\sup}}{\delta_{\inf}}\left(1 + \frac{2}{1-\gamma}\right)\|X^1 - X^2\|_{T,1}\,. \qedhere
\end{align*}
\end{proof}

\begin{remark} \label{rem:lipschitzdiagonalized}
If the system has already been diagonalized, i.e., $\boldsymbol{F}\in\widetilde{R}$, then the results of Theorem \ref{thm:DiagLipschitz} hold with $\delta_{\inf} = \delta_{\sup} = 1$.
\end{remark}

We now show in the next lemma that for fixed $X\in\mathcal{D}_T(L_1)$, $\Psi_F(X)$ and $\Phi_F(X)$ are both Lipschitz in $\boldsymbol{F}\in\widetilde{\mathcal{R}}$. 

\begin{lemma}
\label{lem:refbound}
Let $\boldsymbol{F}_1$ and $\boldsymbol{F}_2$ be two operators in $\widetilde{\mathcal{R}}$ satisfying $\|\boldsymbol{F}_1\|_{op}=\gamma_1<1$ and $\|\boldsymbol{F}_2\|_{op}=\gamma_2<1$. Then for any $X\in \mathcal{D}^T(L_1)$ and all $n\geq1$, 
\begin{enumerate}

\item[(i)] 
\begin{equation}
\label{eq:regbound1}
\|\Psi_{F_i}(X)\|_{T,1} \leq \frac{\|X\|_{T,1}}{1-\gamma_i}\,,
\end{equation}

\item[(ii)]
\begin{equation}
\label{eq:regbound2}
\|\Psi_{F_1}(X)-\Psi_{F_2}(X)\|_{T,1} \leq \frac{\|\boldsymbol{F}_1-\boldsymbol{F}_2\|_{op}\|X\|_{T,1}}{(1-\gamma_1)(1-\gamma_2)}\,.
\end{equation}

\end{enumerate}
\end{lemma}

\begin{proof}
We begin by writing $Y^{i,0} = 0$ and $Y^{i,n+1} = \pi_{X,F_i}(Y^{i,n})$, $i=1,2$. 
\begin{enumerate}
\item[(i)]
Note that
\begin{align*}
\|Y^{i,n+1}\|_{T,1} &= \|\pi_{X,F_i}(Y^{i,n})\|_{T,1}\\
&= \bigg\|\sup_{0\le t\le T}\bigg|\sup_{0\le s\le t}[-X(s) + \boldsymbol{F}_iY^{i,n}(s)]^+\bigg|\bigg\|_1 \\
&\le \bigg\|\sup_{0\le t\le T}\bigg|-X(t) + \boldsymbol{F}_iY^{i,n}(t)\bigg|\bigg\|_1 \\
&\le \|X\|_{T,1} + \|\boldsymbol{F}_i\|_{op}\|Y^{i,n}\|_{T,1}\,.
\end{align*}
Using the reverse triangle inequality, $\big|\|Y^{i,n}\|_{T,1} - \|\Psi_{F_i}(X)\|_{T,1}\big| \le \|Y^{i,n} - \Psi_{F_i}(X)\|_{T,1}$,
and so applying Lemma \ref{lem:pitoreg} yields $\|Y^{i,n}\|_{T,1} \to \|\Psi_{F_i}(X)\|_{T,1}$, and so taking $n\to\infty$ in the above we have
\begin{align*}
\|\Psi_{F_i}(X)\|_{T,1} \le \|X\|_{T,1} + \gamma_i\|\Psi_{F_i}(X)\|_{T,1}.
\end{align*}
Finally, since $\gamma_i<1$, rearranging gives the result.

\item[(ii)]
Using the iteration formula and applying \eqref{eq:supproperty} in the first inequality, we have
\begin{align*}
\|Y^{1,n+1}-Y^{2,n+1}\|_{T,1} &= \|\pi_{X,F_1}(Y^{1,n}) - \pi_{X,F_2}(Y^{2,n})\|_{T,1}\\
&= \bigg\|\sup_{0\le t\le T}\bigg|\sup_{0\le s\le t}[-X(s) + \boldsymbol{F}_1Y^{1,n}(s)]^+ - \sup_{0\le s\le t}[-X(s) + \boldsymbol{F}_2Y^{2,n}(s)]^+\bigg|\bigg\|_1 \\
&\le \bigg\|\sup_{0\le t\le T}\big|[-X(t) + \boldsymbol{F}_1Y^{1,n}(t)]^+ - [-X(t) + \boldsymbol{F}_2Y^{2,n}(t)]^+\big|\bigg\|_1 \\
&\le \bigg\|\sup_{0\le t\le T}\big|\boldsymbol{F}_1Y^{1,n}(t) - \boldsymbol{F}_2Y^{2,n}(t)\big|\bigg\|_1 \\
&\le \|\boldsymbol{F}_1Y^{1,n} - \boldsymbol{F}_2Y^{1,n}\|_{T,1} + \|\boldsymbol{F}_2Y^{1,n} - \boldsymbol{F}_2Y^{2,n}\|_{T,1} \\
&\le \|\boldsymbol{F}_1-\boldsymbol{F}_2\|_{op}\|Y^{1,n}\|_{T,1} + \|\boldsymbol{F}_2\|_{op}\|Y^{1,n} - Y^{2,n}\|_{T,1}\,.
\end{align*}
Again applying the reverse triangle inequality, we obtain  $$\big|\|Y^{1,n}-Y^{2,n}\|_{T,1} - \|\Psi_{F_1}(X)-\Psi_{F_2}(X)\|_{T,1}\big| \le \|Y^{1,n}-\Psi_{F_1}(X)\|_{T,1} + \|\Psi_{F_2}(X)-Y^{2,n}\|_{T,1}\,, $$
and so applying Lemma \ref{lem:pitoreg} yields $\|Y^{1,n}-Y^{2,n}\|_{T,1} \to \|\Psi_{F_1}(X)-\Psi_{F_2}(X)\|_{T,1}$, and so taking $n\to\infty$ in the above we have
\begin{align*}
\|\Psi_{F_1}(X)-\Psi_{F_2}(X)\|_{T,1} \le \|\boldsymbol{F}_1-\boldsymbol{F}_2\|_{op}\|\Psi_{F_1}(X)\|_{T,1} + \gamma_2\|\Psi_{F_1}(X)-\Psi_{F_2}(X)\|_{T,1}\,.
\end{align*}
Finally, noting that $\gamma_2<1$, apply part (i) and rearrange for the result. \qedhere
\end{enumerate}
\end{proof}

We now have the necessary results to show that the original process exhibits Lipschitz-like conditions in both $\boldsymbol{F}$ and $X$ combined. However, since we have proved the above two results for the diagonalized system, in bringing the original processes back we introduce the necessary terms involving the specific diagonal operators for each process. While this may initially seem messy, in the results following this theorem we show that as long as the reflection operators are close enough in the $\|\cdot\|_{op}$ norm, we can in fact choose a common diagonal operator for the diagonalized processes, resulting in the corresponding Lipschitz results.

\begin{theorem}
\label{thm:doublelipschitz}
Take $X^1\in\mathcal{D}_T(L_1)$ and $X^2\in\mathcal{D}_T(L_1)$. Let $(\Phi_{F_1}(X^1),\Psi_{F_1}(X^1))$ and $(\Phi_{F_2}(X^2),\Psi_{F_2}(X^2))$ be the corresponding reflected processes and regulators associated with operators $\boldsymbol{F}_1\in\mathcal{R}$ and $\boldsymbol{F}_2\in\mathcal{R}$ both satisfying Assumption \ref{assumption:generalF}. Let $\boldsymbol{D}_1, \boldsymbol{D}_2$ be strictly positive diagonal operators such that $\|\boldsymbol{D}_i\boldsymbol{F}_i\boldsymbol{D}_i^{-1}\|_{op} = \gamma_i<1$, $\|\boldsymbol{D}_i\|_{op}=\delta^+_i$, and $\|\boldsymbol{D}^{-1}_i\|_{op}=\delta^-_i$, $i=1,2$, and write $\beta_2 \coloneq 1 + \|\boldsymbol{D}_2^{-1}\|_{op}\|\boldsymbol{D}_2\|_{op}$.
Then the following holds:
\begin{enumerate}[label=\roman*]
\item [(i)]
\begin{equation} \label{eq:doublip1}
\begin{split}
&\|\Psi_{F_2}(X^2)-\Psi_{F_1}(X^1)\|_{T,1} \\&\le \frac{\delta^-_2\delta^+_2}{1-\gamma_2}\|X^2-X^1\|_{T,1} \\
&\quad + \frac{\delta^-_2\delta^+_2\delta^-_1\delta^+_1\|X^1\|_{T,1}}{(1-\gamma_1)(1-\gamma_2)}\|\boldsymbol{F}_2-\boldsymbol{F}_1\|_{op} \\
&\quad+ \bigg(\frac{\delta^-_2\|X^1\|_{T,1}}{1-\gamma_2} + \frac{\delta^-_2\delta^-_1\delta^+_1\|\boldsymbol{F}_1\|_{op}\|X^1\|_{T,1}}{(1-\gamma_1)(1-\gamma_2)}\bigg)\|\boldsymbol{D}_2 - \boldsymbol{D}_1\|_{op}\\
&\quad+ \bigg(\frac{\delta^+_1\|X^1\|_{T,1}}{1-\gamma_1} + \frac{\delta^-_2\delta^+_2\delta^+_1\|\boldsymbol{F}_2\|_{op}\|X^1\|_{T,1}}{(1-\gamma_1)(1-\gamma_2)}\bigg)\|\boldsymbol{D}_2^{-1} - \boldsymbol{D}_1^{-1}\|_{op}\,; 
\end{split}
\end{equation}
\item [(ii)]
\begin{equation} \label{eq:doublip2}
\begin{split}
&\|\Phi_{F_2}(X^2)-\Phi_{F_1}(X^1)\|_{T,1} 
\\&\le \bigg(1+\frac{\beta_2\delta^-_2\delta^+_2}{1-\gamma_2}\bigg)\|X^2-X^1\|_{T,1}  \\
&\quad + \bigg(\frac{\beta_2\delta^-_2\delta^+_2\delta^-_1\delta^+_1\|X^1\|_{T,1}}{(1-\gamma_1)(1-\gamma_2)} + \frac{\delta^-_1\delta^+_1\|X^1\|_{T,1}}{1-\gamma_1}\bigg)\|\boldsymbol{F}_2-\boldsymbol{F}_1\|_{op}\\
&\quad+ \beta_2\bigg(\frac{\delta^-_2\|X^1\|_{T,1}}{1-\gamma_2} + \frac{\delta^-_2\delta^-_1\delta^+_1\|\boldsymbol{F}_1\|_{op}\|X^1\|_{T,1}}{(1-\gamma_1)(1-\gamma_2)}\bigg)\|\boldsymbol{D}_2 - \boldsymbol{D}_1\|_{op}\\
&\quad+ \beta_2\bigg(\frac{\delta^+_1\|X^1\|_{T,1}}{1-\gamma_1} + \frac{\delta^-_2\delta^+_2\delta^+_1\|\boldsymbol{F}_2\|_{op}\|X^1\|_{T,1}}{(1-\gamma_1)(1-\gamma_2)}\bigg)\|\boldsymbol{D}_2^{-1} - \boldsymbol{D}_1^{-1}\|_{op} \,.
\end{split}
\end{equation}
\end{enumerate}
\end{theorem}

\begin{proof}
Consider the diagonalized systems using operators $\boldsymbol{\widetilde{P}}_1 = \boldsymbol{D}_1\boldsymbol{F}_1\boldsymbol{D}_1^{-1}$ and $\boldsymbol{\widetilde{P}}_2 = \boldsymbol{D}_2\boldsymbol{F}_2\boldsymbol{D}_2^{-1}$ where $\boldsymbol{D}_1, \boldsymbol{D}_2$ are positive multiplication operators with $\|\boldsymbol{\widetilde{P}}_1\|_{op} = \gamma<1$ and $\|\boldsymbol{\widetilde{P}}_2\|_{op} = \gamma_2<1$. Further let $\widetilde{X}^1=\boldsymbol{D}_1X^1$ and $\widetilde{X}^2=\boldsymbol{D}_2X^2.$

Applying Lemma \ref{lem:refbound} and Theorem \ref{thm:DiagLipschitz} (in particular Remark \ref{rem:lipschitzdiagonalized}) to the diagonalized process, we obtain 
\begin{align*}
\|\Psi_{\widetilde{P}_2}(\widetilde{X}^2)-\Psi_{\widetilde{P}_1}(\widetilde{X}^1)\|_{T,1} &\leq \|\Psi_{\widetilde{P}_2}(\widetilde{X}^2)-\Psi_{\widetilde{P}_2}(\widetilde{X}^1)\|_{T,1} + \|\Psi_{\widetilde{P}_2}(\widetilde{X}^1)-\Psi_{\widetilde{P}_1}(\widetilde{X}^1)\|_{T,1} \\
&\leq \frac{1}{1-\gamma_2}\|\widetilde{X}^2-\widetilde{X}^1\|_{T,1} + \frac{\|\widetilde{X}^1\|_{T,1}}{(1-\gamma_1)(1-\gamma_2)}\|\boldsymbol{\widetilde{P}}_2-\boldsymbol{\widetilde{P}}_1\|_{op}\,.
\end{align*}
Using the relations $\Psi_{F_i}(X^j) = \boldsymbol{D}_i^{-1}\Psi_{\widetilde{P}_i}(\boldsymbol{D}_iX^j)$ and $\Phi_F(X) = \boldsymbol{D}^{-1}\Phi_{\widetilde{P}}(\boldsymbol{DX})$ results in
\begin{align*}
\|\Psi_{F_2}(X^2)-\Psi_{F_1}(X^1)\|_{T,1} &= \|\boldsymbol{D}_2^{-1}\Psi_{\widetilde{P}_2}(\boldsymbol{D}_2X^2) - \boldsymbol{D}_1^{-1}\Psi_{\widetilde{P}_1}(\boldsymbol{D}_1X^1)\|_{T,1}\\
&\le \|\boldsymbol{D}_2^{-1}\Psi_{\widetilde{P}_2}(\boldsymbol{D}_2X^2) - \boldsymbol{D}_2^{-1}\Psi_{\widetilde{P}_1}(\boldsymbol{D}_1X^1)\|_{T,1} \\
&\quad + \|\boldsymbol{D}_2^{-1}\Psi_{\widetilde{P}_1}(\boldsymbol{D}_1X^1) - \boldsymbol{D}_1^{-1}\Psi_{\widetilde{P}_1}(\boldsymbol{D}_1X^1)\|_{T,1} \\
&\le \|\boldsymbol{D}_2^{-1}\|_{op}\|\Psi_{\widetilde{P}_2}(\boldsymbol{D}_2X^2) - \Psi_{\widetilde{P}_1}(\boldsymbol{D}_1X^1)\|_{T,1} \\
&\quad + \|\boldsymbol{D}_2^{-1} - \boldsymbol{D}_1^{-1}\|_{op}\|\Psi_{\widetilde{P}_1}(\boldsymbol{D}_1X^1)\|_{T,1}\\
&\le \frac{\|\boldsymbol{D}_2^{-1}\|_{op}}{1-\gamma_2}\|\boldsymbol{D}_2X^2-\boldsymbol{D}_1X^1\|_{T,1} \\
&\quad+ \frac{\|\boldsymbol{D}_2^{-1}\|_{op}\|\boldsymbol{D}_1X^1\|_{T,1}}{(1-\gamma_1)(1-\gamma_2)}\|\boldsymbol{\widetilde{P}}_2-\boldsymbol{\widetilde{P}}_1\|_{op} \\
&\quad + \|\boldsymbol{D}_2^{-1} - \boldsymbol{D}_1^{-1}\|_{op}\|\Psi_{\widetilde{P}_1}(\boldsymbol{D}_1X^1)\|_{T,1}\,.
\end{align*}
For the first part, we have
\begin{align*}
\frac{\|\boldsymbol{D}_2^{-1}\|_{op}}{1-\gamma_2}\|\boldsymbol{D}_2X^2-\boldsymbol{D}_1X^1\|_{T,1} & = \frac{\|\boldsymbol{D}_2^{-1}\|_{op}}{1-\gamma_2}\|\boldsymbol{D}_2X^2-\boldsymbol{D}_2X^1+\boldsymbol{D}_2X^1-\boldsymbol{D}_1X^1\|_{T,1}\\
&\le \frac{\|\boldsymbol{D}_2^{-1}\|_{op}\|\boldsymbol{D}_2\|_{op}}{1-\gamma_2}\|X^2-X^1\|_{T,1}  \\
&\quad+ \frac{\|\boldsymbol{D}_2^{-1}\|_{op}\|X^1\|_{T,1}}{1-\gamma_2}\|\boldsymbol{D}_2 - \boldsymbol{D}_1\|_{op}\,.
\end{align*}

For the second part, writing $\boldsymbol{D}_2\boldsymbol{F}_2\boldsymbol{D}_2^{-1}-\boldsymbol{D}_1\boldsymbol{F}_1\boldsymbol{D}_1^{-1} = (\boldsymbol{D}_2\boldsymbol{F}_2\boldsymbol{D}_2^{-1}-\boldsymbol{D}_2\boldsymbol{F}_2\boldsymbol{D}_1^{-1}) + (\boldsymbol{D}_2\boldsymbol{F}_2\boldsymbol{D}_1^{-1}-\boldsymbol{D}_2\boldsymbol{F}_1\boldsymbol{D}_1^{-1}) + (\boldsymbol{D}_2\boldsymbol{F}_1\boldsymbol{D}_1^{-1}-\boldsymbol{D}_1\boldsymbol{F}_1\boldsymbol{D}_1^{-1})$, we have 
\begin{align*}
& \frac{\|\boldsymbol{D}_2^{-1}\|_{op}\|\boldsymbol{D}_1X^1\|_{T,1}}{(1-\gamma_1)(1-\gamma_2)}\|\boldsymbol{\widetilde{P}}_2-\boldsymbol{\widetilde{P}}_1\|_{op} \\
&\qquad\le \frac{\|\boldsymbol{D}_2^{-1}\|_{op}\|\boldsymbol{D}_1\|_{op}\|X^1\|_{T,1}}{(1-\gamma_1)(1-\gamma_2)}\|\boldsymbol{D}_2\boldsymbol{F}_2\boldsymbol{D}_2^{-1}-\boldsymbol{D}_1\boldsymbol{F}_1\boldsymbol{D}_1^{-1}\|_{op} \\
&\qquad\le \frac{\|\boldsymbol{D}_2^{-1}\|_{op}\|\boldsymbol{D}_1\|_{op}\|X^1\|_{T,1}}{(1-\gamma_1)(1-\gamma_2)}\bigg(\|\boldsymbol{D}_2\boldsymbol{F}_2\boldsymbol{D}_2^{-1}-\boldsymbol{D}_2\boldsymbol{F}_2\boldsymbol{D}_1^{-1}\|_{op}\\
&\qquad\quad+\|\boldsymbol{D}_2\boldsymbol{F}_2\boldsymbol{D}_1^{-1}-\boldsymbol{D}_2\boldsymbol{F}_1\boldsymbol{D}_1^{-1}\|_{op} + \|\boldsymbol{D}_2\boldsymbol{F}_1\boldsymbol{D}_1^{-1}-\boldsymbol{D}_1\boldsymbol{F}_1\boldsymbol{D}_1^{-1}\|_{op}\bigg) \\
&\qquad\le  \frac{\|\boldsymbol{D}_2^{-1}\|_{op}\|\boldsymbol{D}_2\|_{op}\|\boldsymbol{D}_1\|_{op}\|\boldsymbol{F}_2\|_{op}\|X^1\|_{T,1}}{(1-\gamma_1)(1-\gamma_2)}\|\boldsymbol{D}_2^{-1} - \boldsymbol{D}_1^{-1}\|_{op} \\
&\qquad\quad+  \frac{\|\boldsymbol{D}_2^{-1}\|_{op}\|\boldsymbol{D}_2\|_{op}\|\boldsymbol{D}_1^{-1}\|_{op}\|\boldsymbol{D}_1\|_{op}\|X^1\|_{T,1}}{(1-\gamma_1)(1-\gamma_2)}\|\boldsymbol{F}_2-\boldsymbol{F}_1\|_{op}\\
&\qquad\quad+ \frac{\|\boldsymbol{D}_2^{-1}\|_{op}\|\boldsymbol{D}_1^{-1}\|_{op}\|\boldsymbol{D}_1\|_{op}\|\boldsymbol{F}_1\|_{op}\|X^1\|_{T,1}}{(1-\gamma_1)(1-\gamma_2)}\|\boldsymbol{D}_2 - \boldsymbol{D}_1\|_{op}\,.
\end{align*}

For the third part, using \eqref{eq:regbound1} for the diagonalized system in the first inequality, we have 
\begin{align*}
\|\boldsymbol{D}_2^{-1} - \boldsymbol{D}_1^{-1}\|_{op}\|\Psi_{\widetilde{P}_1}(\boldsymbol{D}_1X^1)\|_{T,1} &\le \|\boldsymbol{D}_2^{-1} - \boldsymbol{D}_1^{-1}\|_{op}\frac{\|\boldsymbol{D}_1X^1\|_{T,1}}{1-\gamma_1} \\
&\le \frac{\|\boldsymbol{D}_1\|_{op}\|X^1\|_{T,1}}{1-\gamma_1}\|\boldsymbol{D}_2^{-1} - \boldsymbol{D}_1^{-1}\|_{op}\,.
\end{align*}
Putting all of these together proves (i).

\medskip

 To show (ii), 
using the relations 
\[\Phi_{F_1}(X^1) = X^1 + (\mathbf{1}-\boldsymbol{F}_1)\Psi_{F_1}(X^1) \quad \text{and} \Phi_{F_2}(X^2) = X^2 + (\mathbf{1}-\boldsymbol{F}_2)\Psi_{F_2}(X^2),\]
we have
$$\Phi_{F_2}(X^2)-\Phi_{F_1}(X^1) = (X^2-X^1)+(\mathbf{1}-\boldsymbol{F}_2)(\Psi_{F_2}(X^2)-\Psi_{F_1}(X^1))+(\boldsymbol{F}_2-\boldsymbol{F}_1)\Psi_{F_1}(X^1)\,.$$
Note that $\|\boldsymbol{F}_2\|_{op} \le \|\boldsymbol{D}_2^{-1}\|_{op}\|\boldsymbol{D}_2\|_{op}\|\boldsymbol{\widetilde{P}}_2\|$ and so $\|{\bf1}-\boldsymbol{F}_2\|_{op} \le 1 + \|\boldsymbol{D}_2^{-1}\|_{op}\|\boldsymbol{D}_2\|_{op} \eqcolon \beta_2$. Taking the norm and using \eqref{eq:regbound1} gives
\begin{align*}
&\|\Phi_{F_2}(X^2)-\Phi_{F_1}(X^1)\|_{T,1} 
\\&\leq \|X^2-X^1\|_{T,1}+\|\mathbf{1}-\boldsymbol{F}_2\|_{op}\|\Psi_{F_2}(X^2)-\Psi_{F_1}(X^1)\|_{T,1}+\|\boldsymbol{F}_2-\boldsymbol{F}_1\|_{op}\|\Psi_{F_1}(X^1)\|_{T,1} \\
&\leq \|X^2-X^1\|_{T,1} + \beta_2\|\Psi_{F_2}(X^2)-\Psi_{F_1}(X^1)\|_{T,1} + \|\boldsymbol{F}_2-\boldsymbol{F}_1\|_{op}\|\boldsymbol{D}_1^{-1}\Psi_{\widetilde{P}_1}(\boldsymbol{D}_1X^1)\|_{T,1} \\
&\le  \|X^2-X^1\|_{T,1} + \beta_2\|\Psi_{F_2}(X^2)-\Psi_{F_1}(X^1)\|_{T,1} + \frac{\|\boldsymbol{D}_1^{-1}\|_{op}\|\boldsymbol{D}_1\|_{op}\|X^1\|_{T,1}}{1-\gamma_1}\|\boldsymbol{F}_2-\boldsymbol{F}_1\|_{op}\,.
\end{align*}
Finally, applying \eqref{eq:doublip1} results in (ii).
\end{proof}

\begin{remark}
In cases where $\boldsymbol{F}$ is motivated by certain constraints, such as the infinite-dimensional analogue of the routing matrix of an open Jackson network, it is common to impose that $\|\boldsymbol{F}\|_{op}\le1$. In this case, we have that $\|{\bf1}-\boldsymbol{F}\|_{op} \le 2$. This simplifies the above theorem by replacing $\beta_2$ with $2$ in \eqref{eq:doublip2}. This is the case that we assume in our previous work  \cite{clarke2026growing}.
\end{remark}

For any operator $\boldsymbol{F}\in\mathcal{R}$, as guaranteed in Proposition \ref{prop:OpDiag} we can always find a diagonal operator $\boldsymbol{D}$ such that $\|\boldsymbol{D}\boldsymbol{F}\boldsymbol{D}^{-1}\|_{op}<1$. However, to obtain Lipschitz results that do not involve differences between these diagonal operators, it is more desirable to diagonalize processes with the same diagonal operator. The next lemma states that this is possible as long as the original operators are close enough to each other.

\begin{lemma}
\label{lem:closeops}
Suppose that $\boldsymbol{F}_1$ and $\boldsymbol{F}_2$ are operators in $\mathcal{R}$ with digaonal operator $\boldsymbol{D}_1$ such that $\|\boldsymbol{D}_1\boldsymbol{F}_1\boldsymbol{D}_1^{-1}\|_{op} = \gamma_1<1$. Then if 
\begin{equation} \label{eq:f1f2close}
\|\boldsymbol{F}_2-\boldsymbol{F}_1\|_{op} < \frac{1-\gamma_1}{\|\boldsymbol{D}_1\|_{op}\|\boldsymbol{D}_1^{-1}\|_{op}}\,, 
\end{equation}
then we also have $\|\boldsymbol{D}_1\boldsymbol{F}_2\boldsymbol{D}_1^{-1}\|_{op}<1$.
\end{lemma}

\begin{proof}
The claim follows from the triangle inequality:
\begin{align*}
\|\boldsymbol{D}_1\boldsymbol{F}_2\boldsymbol{D}_1^{-1}\|_{op} &\le \|\boldsymbol{D}_1\boldsymbol{F}_1\boldsymbol{D}_1^{-1}\|_{op} + \|\boldsymbol{D}_1(\boldsymbol{F}_2-\boldsymbol{F}_1)\boldsymbol{D}_1^{-1}\|_{op} \\
&\le \gamma_1 + \|\boldsymbol{D}_1\|_{op}\|\boldsymbol{F}_2-\boldsymbol{F}_1\|_{op}\|\boldsymbol{D}_1^{-1}\|_{op}\,.
\end{align*}
Hence, if $\|\boldsymbol{F}_2-\boldsymbol{F}_1\|_{op} < \frac{1-\gamma_1}{\|\boldsymbol{D}_1\|_{op}\|\boldsymbol{D}_1^{-1}\|_{op}}$, then $\|\boldsymbol{D}_1\boldsymbol{F}_2\boldsymbol{D}_1^{-1}\|_{op}<1$ as required.
\end{proof}

\begin{remark}
Under Assumption \ref{assumption:generalF}, if $\boldsymbol{D}_1$ is the diagonal matrix associated with $\boldsymbol{F}_1$ from Proposition \ref{prop:OpDiag} (the dominant right eigenfunction of $\boldsymbol{F}^T$), then both $\|\boldsymbol{D}_1\|_{op}$ and $\|\boldsymbol{D}_1^{-1}\|_{op}$ are finite and so there will always be a nontrivial region of operators that satisfy the required inequality in Lemma \ref{lem:closeops}. In this way, Lemma \ref{lem:closeops} states that for two reflection operators close enough to each other, one can use the diagonal operator associated with the other, for example,  the one guaranteed by Proposition \ref{prop:OpDiag}. While this shows that a diagonal operator $\boldsymbol{D}$ associated with a reflection operator $\boldsymbol{F}$ is not unique, it turns out that diagonalization using the dominant eigenfunction provides the smallest possible value of $\|\boldsymbol{D}\boldsymbol{F}\boldsymbol{D}^{-1}\|_{op}$. 
\end{remark}

The above lemma motivates a modification of Theorem \ref{thm:doublelipschitz} that removes the last two terms of \eqref{eq:doublip1} and \eqref{eq:doublip2} allowing for Lipschitz continuity in both the free process and reflection operator.

\begin{theorem}
\label{thm:doublelipschitzsame}
Take $X^1\in\mathcal{D}_T(L_1)$ and $X^2\in\mathcal{D}_T(L_1)$. Let $(\Phi_{F_1}(X^1),\Psi_{F_1}(X^1))$ and $(\Phi_{F_2}(X^2),\Psi_{F_2}(X^2))$ be the corresponding reflected processes and regulators associated with operators $\boldsymbol{F}_1\in\mathcal{R}$ and $\boldsymbol{F}_2\in\mathcal{R}$ both satisfying Assumption \ref{assumption:generalF}. Let $\boldsymbol{D}$ be a positive diagonal operator such that $\|\boldsymbol{D}\boldsymbol{F}_1\boldsymbol{D}^{-1}\|_{op} = \gamma_1<1$ and $\|\boldsymbol{D}\boldsymbol{F}_2\boldsymbol{D}^{-1}\|_{op} = \gamma_2$, and let $\kappa \coloneq \|\boldsymbol{D}^{-1}\|_{op}\|\boldsymbol{D}\|_{op}$.
Then if \eqref{eq:f1f2close} holds,
then $\gamma \coloneq \max(\gamma_1,\gamma_2)<1$ and the following holds:
\begin{enumerate}[label=\roman*]
\item [(i)]
\begin{equation} \label{eq:doublip1cor}
\begin{split}
\|\Psi_{F_2}(X^2)-\Psi_{F_1}(X^1)\|_{T,1} &\le \frac{\kappa}{1-\gamma}\|X^2-X^1\|_{T,1} \\
&\quad + \frac{\kappa^2\|X^1\|_{T,1}}{(1-\gamma)^2}\|\boldsymbol{F}_2-\boldsymbol{F}_1\|_{op}\,,
\end{split}
\end{equation}
\item [(ii)]
\begin{equation} \label{eq:doublip2cor}
\begin{split}
\|\Phi_{F_2}(X^2)-\Phi_{F_1}(X^1)\|_{T,1} 
&\le \bigg(1+\frac{(1+\kappa)\kappa}{1-\gamma_1}\bigg)\|X^2-X^1\|_{T,1}  \\
&\quad + \bigg(\frac{(1+\kappa)\kappa^2\|X^1\|_{T,1}}{(1-\gamma_1)^2} + \frac{\kappa\|X^1\|_{T,1}}{1-\gamma_1}\bigg)\|\boldsymbol{F}_2-\boldsymbol{F}_1\|_{op}\,.
\end{split}
\end{equation}
\end{enumerate}
\end{theorem}

\begin{proof}
$\gamma_2<1$ follows from Lemma \ref{lem:closeops}. Applying this to Theorem \ref{thm:doublelipschitz} completes the proof. 
\end{proof}

In Theorem \ref{thm:doublelipschitz} we have shown that given any free processes $X_1$ and $X_2$ in $\mathcal{D}_T(L_1)$, and reflection operators $\boldsymbol{F}_1$, $\boldsymbol{F}_2$ in $\mathcal{R}$ with respective diagonal operators $\boldsymbol{D}_1$ and $\boldsymbol{D}_2$, both the regulator and reflection maps have Lipschitz-like properties, with the addition of terms involving the distance between $\boldsymbol{D}_1$ and $\boldsymbol{D}_2$. Theorem \ref{thm:doublelipschitzsame} eliminates these terms by showing that for $\boldsymbol{F}_1$ and $\boldsymbol{F}_2$ close enough we can choose a common diagonal operator. We now extend these two ideas to a sequence of reflection operators $\{\boldsymbol{F}_n\}_{n\geq1}\in\mathcal{R}$ converging to some $\boldsymbol{F}\in\mathcal{R}$. With the aid of the following lemma, we will show in Theorem \ref{thm:seqlipschitz} that for sufficiently large $n$, the regulator and reflected processes will be Lipschitz in both the free process and the reflection operator.

\begin{lemma}
\label{lem:diagsequence}
Take $\boldsymbol{F}\in\mathcal{R}$ satisfying Assumption \ref{assumption:generalF} and let $\boldsymbol{D}$ be a positive multiplication operator such that $\|\boldsymbol{D}\boldsymbol{F}\boldsymbol{D}^{-1}\|_{op}=\gamma<1$ as in Proposition \ref{prop:OpDiag}. Consider a sequence $\{\boldsymbol{F}_n\}_{n\geq1}\in\mathcal{R}$ such that $\|\boldsymbol{F}_n-\boldsymbol{F}\|_{op}\to0$ as $n\to\infty$. Then 
$$\|\boldsymbol{D}\boldsymbol{F}_n\boldsymbol{D}^{-1}\|_{op}\eqcolon\gamma_n\to\gamma \quad \mbox{as }n\to\infty\,.$$
In particular, for sufficiently large $n$,
$$\gamma_n<1\,.$$
\end{lemma}

\begin{proof}
The first claim follows from the argument below:

\begin{align*}
\left| \|\boldsymbol{D}\boldsymbol{F}_n\boldsymbol{D}^{-1}\|_{op} - \|\boldsymbol{D}\boldsymbol{F}\boldsymbol{D}^{-1}\|_{op} \right| &\le \|\boldsymbol{D}(\boldsymbol{F}_n-\boldsymbol{F})\boldsymbol{D}^{-1}\|_{op} \\
&\le \|\boldsymbol{D}\|_{op}\|\boldsymbol{F}_n-\boldsymbol{F}\|_{op}\|\boldsymbol{D}^{-1}\|_{op}\,,
\end{align*}
and hence 
$\|\boldsymbol{D}\boldsymbol{F}_n\boldsymbol{D}^{-1}\|_{op} \to \|\boldsymbol{D}\boldsymbol{F}\boldsymbol{D}^{-1}\|_{op} = \gamma$ as $n\to\infty$.
The second claim is subsequently immediate.
\end{proof}

\begin{theorem} \label{thm:seqlipschitz}
Take $X\in\mathcal{D}_T(L_1)$ and let $\{X^n\}_{n\geq1}$ be a sequence in $\mathcal{D}_T(L_1)$. Let $(\Phi_F(X),\Psi_F(X))$ and $\left(\{\Phi_{F_n}(X_n)\}_{n\geq1},\{\Psi_{F_n}(X_n)\}_{n\geq1}\right)$ be the corresponding reflected processes and regulators associated with operators $\boldsymbol{F}\in\mathcal{R}$ and $\{\boldsymbol{F}_n\}_{n\geq1}\in\mathcal{R}$ satisfying Assumption \ref{assumption:generalF} respectively. Let $\boldsymbol{D}$ be the positive diagonal operator associated with the diagonalized process of $\boldsymbol{F}$ and let $\kappa \coloneq \|\boldsymbol{D}^{-1}\|_{op}\|\boldsymbol{D}\|_{op}$. If $\|\boldsymbol{F_n}-\boldsymbol{F}\|_{op} \to 0$, then for some $\hat{\gamma}\in(0,1)$ and sufficiently large $n$,
\begin{enumerate}[label=\roman*]
\item [(i)]
\begin{equation}
\|\Psi_{F_n}(X^n)-\Psi_F(X)\|_{T,1} \leq \frac{\kappa}{1-\hat{\gamma}}\|X^n-X\|_{T,1} + \frac{\kappa^2\|X\|_{T,1}}{(1-\hat{\gamma})^2}\|\boldsymbol{F}_n-\boldsymbol{F}\|_{op}\,,
\end{equation}
\item [(ii)]
\begin{equation}
\begin{split}
\|\Phi_{F_n}(X^n)-\Phi_F(X)\|_{T,1} 
&\leq \left(1 +\frac{(1+\kappa)\kappa}{1-\hat{\gamma}}\right)\|X^n-X\|_{T,1} \\&+ \kappa\left(\frac{(1+\kappa)\kappa}{(1-\hat{\gamma})^2} + \frac{1}{1-\hat{\gamma}}\right)\|X\|_{T,1}\|\boldsymbol{F}_n-\boldsymbol{F}\|_{op}\,.
\end{split}
\end{equation}
\end{enumerate}
In particular, if $\|X^n-X\|_{T,1}\to0$ as $n\to\infty$, then
\begin{equation}    
    \|\Psi_{F_n}(X^n)-\Psi_F(X)\|_{T,1} \to 0 \quad\mbox{and}\quad\|\Phi_{F_n}(X^n)-\Phi_F(X)\|_{T,1} \to 0\,.
\end{equation}
\end{theorem}

\begin{proof}
Consider the diagonalized systems using operators $\boldsymbol{\widetilde{P}} = \boldsymbol{D}\boldsymbol{F}\boldsymbol{D}^{-1}$ and $\boldsymbol{\widetilde{P}}_n = \boldsymbol{D}\boldsymbol{F}_n\boldsymbol{D}^{-1}$ where $\boldsymbol{D}$ is a positive multiplication operator with $\|\boldsymbol{\widetilde{P}}\|_{op} = \gamma<1$ and $\|\boldsymbol{\widetilde{P}}_n\|_{op} = \gamma_n<1$ for sufficiently large $n$, say $n\geq N$ by Lemma \ref{lem:diagsequence}. We also note that $\sup_{n\geq N}\gamma_n < 1$ since $\gamma_n\to\gamma$ as $n\to\infty$. Let $\hat{\gamma} = \max(\gamma, \sup_{n\geq N}\gamma_n)$. 
Finally, noticing that $\frac{1}{1-\gamma} \le \frac{1}{1-\hat{\gamma}}$ and $\frac{1}{1-\gamma_n} \le \frac{1}{1-\hat{\gamma}}$ for $n\ge N$, we can apply Theorem \ref{thm:doublelipschitz} for the result.
\end{proof}

\begin{remark}
Theorem \ref{thm:seqlipschitz} shows convergence of systems where both the free process and the reflection operator converge, stating that the regulator and reflected process will also converge. This is particularly useful when dealing with convergence from a finite to an infinite-dimensional system, where the finite-dimensional system can be represented by blockwise-constant processes and operators. This is the case in our previous work \cite{clarke2026growing}. Because the continuity is Lipschitz in both the free process and the reflection operator, if one can show that both of these converge at a particular rate, then one can also derive rates of convergence for the regulator and reflection map.
\end{remark}

\subsection{Lipschitz Continuity in the $J_1$ topology} \label{subsec:skorlip}

While the uniform metric $\|\cdot\|_{T,1}$ provides a suitable framework for establishing existence, uniqueness, and Lipschitz continuity of the Skorokhod map, in many cases it may be too restrictive when considering convergence of càdlàg processes. In many applications (such as queueing and stochastic networks, chemical reaction networks, biological and population systems),
one is interested in sequences of processes whose jump magnitudes converge, but whose jump times may exhibit small perturbations. The uniform topology does not accommodate for such cases, whereas the Skorokhod $J_1$ topology is specifically designed to identify paths that are close up to a small deformation of the time parameter.
Hence, we establish analogues of the Lipschitz continuity properties in Section \ref{subsec:uniflip} under a Skorokhod-type topology adapted to the present infinite-dimensional setting.

Let $\Lambda:[0,T]\to[0,T]$ be the space of all increasing and onto right-continuous functions. For a slight abuse of notation, when $f\in\mathcal{D}_T(L_1)$ and $\lambda\in\Lambda$ we define the composition $f\circ\lambda$ by $(f\circ\lambda)_u(t) = f_u(\lambda(t))$ for each $u \in [0,1]$. 
Define the metric
\begin{equation} \label{eq:skotop}
d_{T,J_1} (f,g)= \inf_{\lambda\in\Lambda} \Big\{\|f\circ\lambda-g\|_{T,1}\lor\sup_{0\le t\le T}|\lambda(t)-t|\Big\}\,.
\end{equation}

Before proceeding, we first present a proposition that provides solutions to the infinite-dimesional Skorokhod problem under rescaling.

\begin{prop}[Rescaling property of the reflection map]
\label{prop:rescale}
Take any
$X\in \mathcal{D}_T(L_1)$ and $\boldsymbol{F}\in\mathcal{R}$ satisfying Assumption \ref{assumption:generalF},
with any $w\in L_1[0,1]$, $\alpha>0$. Let $\gamma$ be a nondecreasing right-continuous map from $[0,T']$ to $[0,T]$. Then
\begin{equation}
\Psi_F(w + \alpha (X\circ\gamma)) = \alpha\Psi_F(\alpha^{-1}w+X)\circ\gamma
\end{equation}
\begin{equation}
\Phi_F(w + \alpha (X\circ\gamma)) = \alpha\Phi_F(\alpha^{-1}w+X)\circ\gamma
\end{equation}
\end{prop}

\begin{proof}
First note that $\alpha^{-1}w+X\in\mathcal{D}_T(L_1)$. Applying Theorem \ref{thm:ExtUni}, let $
Y=\Psi_F(\alpha^{-1}w+X)$ and define $Q=\Phi_F(\alpha^{-1}w+X)$.
Then
$$Q(t)=\alpha^{-1}w+X(t)+({\bf1}-\boldsymbol{F})Y(t),$$
where $Q_u(t)\geq 0$ and $Q_u(t)dY_u(t) = 0$ for all $(u,t)\in([0,1]\times[0,T])$.
For $t\in[0,T']$, define
\begin{align*}
&\widetilde{Y}(t)=\alpha Y(\gamma(t)), \\
&\widetilde{Q}(t)=\alpha Q(\gamma(t)).
\end{align*}

Since $Y$ and $\gamma$ are nondecreasing and right-continuous and $\alpha>0$, it follows that $\widetilde{Y}$ is also nondecreasing and right-continuous. Moreover, $\widetilde{Y}(0)=0$ and $\widetilde{Q}(t)\geq 0$ for all $t\in[0,T']$.

Furthermore, as $\boldsymbol{F}$ is a linear operator,
$$
\begin{aligned}
\widetilde{Q}(t)
&=\alpha\left(\alpha^{-1}w+X(\gamma(t))+({\bf1}-\boldsymbol{F})Y(\gamma(t))\right)\\
&=w+\alpha(X\circ\gamma)(t)+({\bf1}-\boldsymbol{F})\widetilde{Y}(t),.
\end{aligned}
$$

Finally, $\widetilde{Q}_u(t)d\widetilde{Y}_u(t)$ for all $(u,t)\in([0,1]\times[0,T])$, since positive scalar multiplication together with composition by a nondecreasing function preserves the complementarity property.

It is clear by construction that $w + \alpha (X\circ\gamma) \in \mathcal{D}_{T'}(L_1)$. Hence, $(\widetilde{Q},\widetilde{Y})$ is a solution of the Skorokhod problem corresponding to the free process $w+\alpha(X\circ\gamma)$. By uniqueness of solutions to the Skorokhod problem guaranteed in Theorem \ref{thm:ExtUni}, this completes the proof.
\end{proof}

\begin{remark}
While the above proposition holds for any general time rescaling, to proceed with the following Lipschitz results we will only apply the case when $w=0$, $\alpha=1$, and $\gamma\in\Lambda$. The first result, for a shared transition operator, is stated below.
\end{remark}

\begin{theorem} \label{thm:skolipfree}
Under the same conditions as Theorem \ref{thm:DiagLipschitz},
\begin{equation*}
d_{T,J_1} (\Psi_F(X^1),\Psi_F(X^2)) \le \frac{\delta_{\sup}}{\delta_{\inf}(1-\gamma)}d_{T,J_1} (X^1,X^2),
\end{equation*}
\begin{equation*}
d_{T,J_1} (\Phi_F(X^1),\Phi_F(X^2)) \le \frac{\delta_{\sup}}{\delta_{\inf}}\left(1 + \frac{2}{1-\gamma}\right)d_{T,J_1} (X^1,X^2).
\end{equation*}
\end{theorem}

\begin{proof}
Applying Theorem \ref{thm:DiagLipschitz}, we have
\begin{align*}
d_{T,J_1} (\Psi_F(X^1),\Psi_F(X^2))&= \inf_{\lambda\in\Lambda}\Big\{\|\Psi_F(X^1)\circ\lambda-\Psi_F(X^2)\|_{T,1}\lor\sup_{0\le t\le T}|\lambda(t)-t|\Big\}\\
&=\inf_{\lambda\in\Lambda}\Big\{\|\Psi_F(X^1\circ\lambda)-\Psi_F(X^2)\|_{T,1}\lor\sup_{0\le t\le T}|\lambda(t)-t|\Big\}\\
&\le\inf_{\lambda\in\Lambda}\Big\{\frac{\delta_{\sup}}{\delta_{\inf}(1-\gamma)} \|X^1\circ\lambda-X^2\|_{T,1}\lor\sup_{0\le t\le T}|\lambda(t)-t|\Big\}\\
&\le \frac{\delta_{\sup}}{\delta_{\inf}(1-\gamma)}d_{T,J_1} (X^1,X^2).
\end{align*}
The second result follows similarly.
\end{proof}

\begin{theorem} \label{thm:skolipboth}
Under the same conditions as Theorem \ref{thm:seqlipschitz},
\begin{equation*}
d_{T,J_1} (\Psi_{F_n}(X^n),\Psi_F(X)) \le \frac{\kappa}{1-\hat{\gamma}}d_{T,J_1} (X^n,X) + \frac{\kappa^2\|X\|_{T,1}}{(1-\hat{\gamma})^2}\|\boldsymbol{F}_n-\boldsymbol{F}\|_{op},
\end{equation*}
\begin{equation*}
\begin{split}
d_{T,J_1} (\Phi_{F_n}(X^n),\Phi_F(X)) &\le \left(1 +\frac{(1+\kappa)\kappa}{1-\hat{\gamma}}\right)d_{T,J_1} (X^n,X) \\
&+ \kappa\left(\frac{(1+\kappa)\kappa}{(1-\hat{\gamma})^2} + \frac{1}{1-\hat{\gamma}}\right)\|X\|_{T,1}\|\boldsymbol{F}_n-\boldsymbol{F}\|_{op}.
\end{split}
\end{equation*}
\end{theorem}

\begin{proof}
Applying Theorem \ref{thm:seqlipschitz} and choosing $\lambda(t)\equiv t$ in the second part of the first inequality, we have
\begin{align*}
d_{T,J_1} (\Psi_{F_n}(X^n),\Psi_F(X))&= \inf_{\lambda\in\Lambda}\Big\{\|\Psi_{F_n}(X^n)\circ\lambda-\Psi_F(X)\|_{T,1}\lor\sup_{0\le t\le T}|\lambda(t)-t|\Big\}\\
&= \inf_{\lambda\in\Lambda}\Big\{\|\Psi_{F_n}(X^n\circ\lambda)-\Psi_F(X)\|_{T,1}\lor\sup_{0\le t\le T}|\lambda(t)-t|\Big\}\\
&\le \inf_{\lambda\in\Lambda}\Big\{\frac{\kappa}{1-\hat{\gamma}}\|X^n\circ\lambda-X\|_{T,1}\lor\sup_{0\le t\le T}|\lambda(t)-t|\Big\} \\
&\quad+ \inf_{\lambda\in\Lambda}\Big\{\frac{\kappa^2\|X\|_{T,1}}{(1-\hat{\gamma})^2}\|\boldsymbol{F}_n-\boldsymbol{F}\|_{op}\lor\sup_{0\le t\le T}|\lambda(t)-t|\Big\} \\
&\le \frac{\kappa}{1-\hat{\gamma}}d_{T,J_1} (X^n,X) + \frac{\kappa^2\|X\|_{T,1}}{(1-\hat{\gamma})^2}\|\boldsymbol{F}_n-\boldsymbol{F}\|_{op}.
\end{align*}
The second result follows similarly.
\end{proof}

\begin{remark}
While Proposition \ref{prop:rescale} and Theorems \ref{thm:skolipfree} and \ref{thm:skolipboth} imposed Assumption \ref{assumption:generalF} on the reflection operators, this additional assumption only follows from its use in proving results on existence, uniqueness, and Lipschitz continuity via the diagonalization method of this paper. In fact, if one instead adopts the approach used in \cite{clarke2026growing}, existence, uniqueness, and Lipschitz continuity can be shown without Assumption \ref{assumption:generalF} required. It therefore follows that Proposition \ref{prop:rescale} and Theorems \ref{thm:skolipfree} and \ref{thm:skolipboth} all hold true without the additional Assumption \ref{assumption:generalF}.
\end{remark}

\allowdisplaybreaks

\bigskip 
\appendix
\section{Additional Proof}
\label{sec:appendix}

\begin{proof}[Proof of Lemma \ref{lem:doeblintype}]

We first prove eventual strong positivity. Under the conditions of Lemma \ref{lem:doeblintype} there exist $m \in \NN$, $\varepsilon > \frac{1}{2}$, and a density $\nu > 0$ a.e.\ such that 
$$(F^T)^{(m)}(u,v) \ge \varepsilon \nu(v).$$
Let $f\in L_\infty([0,1])$ where $f\ge0$ $f>0$ on a set of positive measure. Then
$$((\boldsymbol{F}^T)^{(m)} f)(u)
= \int_0^1 (F^T)^{(m)}(u,v)\,f(v)\,dv
\ge \varepsilon \int_0^1 \nu(v) f(v)\,dv.$$
Since $\nu > 0$ a.e., we have
$$\operatorname*{ess\,inf}_{u\in[0,1]}((\boldsymbol{F}^T)^{(m)}f)(u) > 0,$$
thus satisfying (i) of Assumption \ref{assumption:generalF}.

\medskip

\noindent To prove quasi-compactness, first note that for each $u,v$, we can write
$$(F^T)^{(m)}(u,v) = \varepsilon \nu(v) + R(u,v),
$$
where 
$$ R(u,v) := (F^T)^{(m)}(u,v) - \varepsilon \nu(v) \ge 0.$$
Define the corresponding operators $\boldsymbol{K}$ and $\boldsymbol{R}$ acting on $L_\infty([0,1])$ through
$$(\boldsymbol{K}f)(u) := \varepsilon \int_0^1 \nu(v) f(v)\,dv, 
\quad (\boldsymbol{R}f)(u) := \int_0^1 R(u,v) f(v)\,dv.$$
It follows that $(\boldsymbol{F}^T)^{(m)} = \boldsymbol{K} + \boldsymbol{R}$. As $\boldsymbol{K}$ is a rank-one operator, it is hence compact. Moreover, as $\nu$ is a probability density,
$$\int_0^1 R(u,v)\,dv
= \int_0^1 (F^T)^m(u,v)\,dv - \varepsilon \int_0^1 \nu(v)\,dv
\le 1 - \varepsilon,$$
and so $\rho_{ess}(\boldsymbol{R})\le\|\boldsymbol{R}\|_\infty \le 1-\epsilon$. Applying \cite[Lemma 4.8]{mallet2010generalizing}, we have $\rho_{ess}(\boldsymbol{F}^{(m)})\le1-\epsilon$.
Using \eqref{eq:doeblin} and taking $L_\infty([0,1])\ni f\equiv1$, we have $(\boldsymbol{F}^T)^{(m)}f \ge \epsilon f$ and so $\rho(\boldsymbol{F}^{(m)}) \ge \epsilon$.
Since $\epsilon>\frac{1}{2}$,  we thus have $\rho_{ess}(\boldsymbol{F}^{(m)})<\rho(\boldsymbol{F}^{(m)})$. From \cite[Theorem 3.7]{mallet2010generalizing},  we have that $\rho_{ess}(\boldsymbol{F}^{(m)}) = \rho_{ess}(\boldsymbol{F})^m$ and since $\rho(\boldsymbol{F}^{(m)}) = \rho(\boldsymbol{F})^m$, it follows that $\rho_{ess}(\boldsymbol{F})<\rho(\boldsymbol{F})$, thus satisfying (ii) of Assumption \ref{assumption:generalF}.
\end{proof}

\bigskip


\end{document}